\documentclass[12pt]{amsart}%
\usepackage{amsfonts}
\usepackage{amsmath}
\usepackage{amssymb}
\usepackage{graphicx}%
\usepackage[margin=1.2in]{geometry}

\newtheorem{theorem}{Theorem}[section]
\newtheorem{lemma}[theorem]{Lemma}

\newtheorem{proposition}[theorem]{Proposition}

\theoremstyle{definition}
\newtheorem{definition}[theorem]{Definition}
\newtheorem{example}[theorem]{Example}
\theoremstyle{remark}

\numberwithin{equation}{section}

\begin{document}
\title{The Fourier transform of thick distributions}
\author{Ricardo Estrada}
\address{R. Estrada, Department of Mathematics\\
Louisiana State University\\
Baton Rouge, LA 70803\\
U.S.A.}
\email{restrada@math.lsu.edu}
\author{Jasson Vindas}
\address{J. Vindas, Department of Mathematics: Analysis, Logic and Discrete Mathematics\\
Ghent University\\
Krijgslaan 281, Building S8\\
B 9000 Ghent, Belgium}
\email{jasson.vindas@ugent.be}
\author{Yunyun Yang}
\address{Y. Yang, School of Mathematics\\
Hefei University of Technology\\
Hefei 230009, China}
\email{yangyunyun@hfut.edu.cn}
\thanks{J. Vindas was supported by Ghent University through the BOF-grants 01J11615
and 01J04017.}
\thanks{Y. Yang was supported through the grant 407-0371000086 from Hefei University
of Techonology.}

\begin{abstract}
We first construct a space $\mathcal{W}\left(  \mathbb{R}_{\text{c}}%
^{n}\right)  $ whose elements are test functions  defined in $\mathbb{R}%
_{\text{c}}^{n}=\mathbb{R}^{n}\cup\left\{  \mathbf{\infty}\right\}  ,$ the one
point compactification of $\mathbb{R}^{n},$ that have a thick expansion at
infinity of special logarithmic type, and its dual space $\mathcal{W}^{\prime
}\left(  \mathbb{R}_{\text{c}}^{n}\right)  ,$ the space of $sl-$thick
distributions. We show that there is a canonical projection of $\mathcal{W}%
^{\prime}\left(  \mathbb{R}_{\text{c}}^{n}\right)  $ onto $\mathcal{S}%
^{\prime}\left(  \mathbb{R}^{n}\right)  .$ We study several $sl-$thick
distributions and consider operations in $\mathcal{W}^{\prime}\left(
\mathbb{R}_{\text{c}}^{n}\right)  .$

We define and study the Fourier transform of thick test functions of
$\mathcal{S}_{\ast}\left(  \mathbb{R}^{n}\right)  $ and thick tempered
distributions of $\mathcal{S}_{\ast}^{\prime}\left(  \mathbb{R}^{n}\right)  .$ We
construct isomorphisms
\[
\mathcal{F}_{\ast}:\mathcal{S}_{\ast}^{\prime}\left(  \mathbb{R}^{n}\right)
\longrightarrow\mathcal{W}^{\prime}\left(  \mathbb{R}_{\text{c}}^{n}\right)
\,,
\]%
\[
\mathcal{F}^{\ast}:\mathcal{W}^{\prime}\left(  \mathbb{R}_{\text{c}}%
^{n}\right)  \longrightarrow\mathcal{S}_{\ast}^{\prime}\left(  \mathbb{R}%
^{n}\right)  \,,
\]
that extend the Fourier transform of tempered distributions, namely,
$\Pi\mathcal{F}_{\ast}=\mathcal{F}\Pi$ and $\Pi\mathcal{F}^{\ast}%
=\mathcal{F}\Pi,$ where $\Pi$ are the canonical projections of $\mathcal{S}%
_{\ast}^{\prime}\left(  \mathbb{R}^{n}\right)  $ or $\mathcal{W}^{\prime
}\left(  \mathbb{R}_{\text{c}}^{n}\right)  $ onto $\mathcal{S}^{\prime}\left(
\mathbb{R}^{n}\right)  .$

We determine the Fourier transform of several finite part regularizations and
of general thick delta functions.

\end{abstract}
\maketitle

\section{Introduction\label{Section: Introduction}}

The aim of this article is to construct the Fourier transform of \emph{thick
tempered distributions }in several variables. Thick distributions were
introduced in one variable in \cite{EF07} and in several variables in
\cite{YE2, YE3, YEnovisad, YE5}. Thick distributions have found applications
in understanding problems in several areas, such as quantum field theory
\cite{BondurantFulling}, engineering \cite{Paskusz, Vibet}, the understanding
of singularities in mathematical physics as considered in \cite{Blinder} or in
\cite{Bowen}, or in obtaining formulas for the regularization of multipoles
\cite{Estrada16, Parker} that play a fundamental role in the ideas of the late
professor Stora on convergent Feyman amplitudes \cite{NST, Varily-Gracia}.
They also appear in other problems, as generalizations of\ Frahm formulas
\cite{Frahm} involving discontinuous test functions \cite{Franklin, YE3}.
Thick distributions are the distributional theory corresponding to the theory
given by Blanchet and Faye \cite{BlanchetFaye}, whose aim is the study of the
dynamics of point particles in high post-Newtonian approximations of general
relativity \cite{Blanchet-Faye-Nissanke} and who develop such a scheme in the
context of finite parts, pseudo-functions and Hadamard regularization, as
studied by Sellier \cite{Sellier2, Sellier3}. In this article we consider
spaces with one thick point, located at the origin, but it is possible to
consider spaces with a finite number of such singular points.

The Fourier transform of one-dimensional thick distributions with one special
point at the origin was given in \cite{EF07}. The transform of thick
distributions is shown to belong to a space $\mathcal{W}^{\prime}$ of
distributions on the space $\mathbb{R}_{\text{c}}=\mathbb{R\cup}\left\{
\infty\right\}  ,$ the one point compactification of the real line.\ Employing
this thick Fourier transform it is possible to understand several puzzles,
particularly those found in \cite{BondurantFulling}.

The theory of thick distributions in higher dimensions \cite{YE2}\ is quite
different from that in one dimension, because the topology of $\mathbb{R}%
^{n}\setminus\left\{  \mathbf{0}\right\}  ,$ $n\geq2,$ is quite unlike that of
$\mathbb{R}\setminus\left\{  0\right\}  ,$ since the latter space is
disconnected, consisting of two unrelated rays, while the former is connected,
all directions of approach to the point $\mathbf{0}$ are related, and such
behavior imposes strong restrictions on the singularities. Therefore the thick
Fourier transform in several variables cannot be constructed as a
straightforward extension of the transform in one variable; such construction
in several variables, the Fourier transform in $\mathcal{S}_{\ast}^{\prime
}\left(  \mathbb{R}^{n}\right)  ,$ is the main aim of this article.

In Section \ref{Section:Preliminaries} we review some useful results from the
theory of thick distributions and then in Section
\ref{Section: Some Fourier Transforms} we collect the Fourier transform of
several tempered distributions in order to find \ the asymptotic expansion of
the Fourier transform of finite part regularizations of thick test functions.
Taking into account the asymptotic behavior of such Fourier transforms we
construct a space $\mathcal{W}\left(  \mathbb{R}_{\text{c}}^{n}\right)  $
whose elements are test functions defined in $\mathbb{R}_{\text{c}}%
^{n}=\mathbb{R}^{n}\cup\left\{  \mathbf{\infty}\right\}  ,$ the one point
compactification of $\mathbb{R}^{n},$ that have a thick expansion at infinity
of special logarithmic type. We are thus able in Section
\ref{Section: The Fourier transform of thick test functions}\ to define
Fourier transform operators $\mathcal{F}_{\ast,\text{t}}$\ and $\mathcal{F}%
_{\text{t}}^{\ast},$ topological isomorphism of $\mathcal{S}_{\ast}\left(
\mathbb{R}^{n}\right)  $ to $\mathcal{W}\left(  \mathbb{R}_{\text{c}}%
^{n}\right)  $ and from $\mathcal{W}\left(  \mathbb{R}_{\text{c}}^{n}\right)
$ to $\mathcal{S}_{\ast}\left(  \mathbb{R}^{n}\right)  ,$ respectively; the
subscript `t' is used because these are the transforms of test functions.

We study the dual space $\mathcal{W}^{\prime}\left(  \mathbb{R}_{\text{c}}%
^{n}\right)  ,$ the space of $sl-$thick distributions in Section
\ref{Section: The space W'}. We consider the basic operations in
$\mathcal{W}^{\prime}\left(  \mathbb{R}_{\text{c}}^{n}\right)  ,$ such as
linear changes of variables, derivatives, and multiplication by polynomials.
We study several $sl-$thick distributions, particularly the finite part
regularization at infinity of power functions and thick delta functions at
infinity. We are therefore able in Section
\ref{Section: The Fourier transform of thick distributions}\ to define and
study the Fourier transform of thick test tempered distributions of
$\mathcal{S}_{\ast}^{\prime}\left(  \mathbb{R}^{n}\right)  .$ We construct
isomorphisms
\[\mathcal{F}_{\ast}:\mathcal{S}_{\ast}^{\prime}\left(  \mathbb{R}^{n}\right)
\longrightarrow\mathcal{W}^{\prime}\left(  \mathbb{R}_{\text{c}}^{n}\right)
\,,
\end{equation*}
\[
\mathcal{F}^{\ast}:\mathcal{W}^{\prime}\left(  \mathbb{R}_{\text{c}}%
^{n}\right)  \longrightarrow\mathcal{S}_{\ast}^{\prime}\left(  \mathbb{R}%
^{n}\right)  \,, 
\]
that extend the Fourier transform of tempered distributions, namely,
\[
\Pi_{\mathcal{W}^{\prime},\mathcal{S}^{\prime}}\mathcal{F}_{\ast}%
=\mathcal{F}\Pi_{\mathcal{S}_{\ast}^{\prime},\mathcal{S}^{\prime}%
}\,,\ \ \ \ \ \ \Pi_{\mathcal{S}_{\ast}^{\prime},\mathcal{S}^{\prime}%
}\mathcal{F}^{\ast}=\mathcal{F}\Pi_{\mathcal{W}^{\prime},\mathcal{S}^{\prime}%
}\,,
\]
where $\Pi_{\mathcal{W}^{\prime},\mathcal{S}^{\prime}}$ and $\Pi
_{\mathcal{S}_{\ast}^{\prime},\mathcal{S}^{\prime}}$ are the canonical
projections of $\mathcal{S}_{\ast}^{\prime}\left(  \mathbb{R}^{n}\right)  $ or
$\mathcal{W}^{\prime}\left(  \mathbb{R}_{\text{c}}^{n}\right)  $ onto
$\mathcal{S}^{\prime}\left(  \mathbb{R}^{n}\right)  .$ We give the
transformation rules for the Fourier transform of derivatives,
multiplications, and linear changes of variables, as well as the Fourier
inversion formulas. We determine the Fourier transform of several finite part
regularizations and of general thick delta functions.

Since we need to employ many spaces, operators, and distributions, following
advise of one referee, we also included an appendix that lists the many
notations used.

\section{Preliminaries\label{Section:Preliminaries}}

We shall use basic facts about distributions and general functional analysis
as can be found in the textbooks \cite{GreenBook, Horvath, Kanwal, Schwartz,
Treves}. However, in this section we recall several recently introduced or not
so well known ideas that will be needed in our analysis. We will also fix the
notation employed; particularly, we use the Fourier transform%
\begin{equation*}
F\left\{  f\left(  \mathbf{x}\right)  ;\mathbf{u}\right\}  =\int
_{\mathbb{R}^{n}}f\left(  \mathbf{x}\right)  e^{i\mathbf{x\cdot u}%
}\,\mathrm{d}\mathbf{x}\,. 
\end{equation*}
We will write
\begin{equation}
c_{m,n}=\frac{2\Gamma\left(  m+1/2\right)  \pi^{\left(  n-1\right)  /2}%
}{\Gamma\left(  m+n/2\right)  }=\int_{\mathbb{S}}\omega_{j}^{2m}%
\,\mathrm{d}\sigma\left(  \mathbf{\omega}\right)  \,,\ \ \ C=c_{0,n}\,.
\label{ce}
\end{equation}
Notice that $c_{0,n}=C=2\pi^{n/2}/\Gamma\left(  n/2\right)  ,$ is the surface
area of the unit sphere $\mathbb{S}$ of $\mathbb{R}^{n},$ denoted as $C_{n-1}$
in \cite{YE2}.

\subsection{Spaces of thick test functions and spaces of thick
distributions\label{Subsection:Spaces of thick test functions and spaces of thick distributions}%
}

The construction of the space of thick test functions $\mathcal{D}%
_{\ast,\mathbf{a}}\left(  \mathbb{R}^{n}\right)  $ and its dual,
$\mathcal{D}_{\ast,\mathbf{a}}^{\prime}\left(  \mathbb{R}^{n}\right)  ,$ the
space of thick distributions is as follows \cite{YE2}. Let $\mathbf{a}$ be a
fixed point of $\mathbb{R}^{n}.$ Let $\mathcal{D}_{\ast,\mathbf{a}}\left(
\mathbb{R}^{n}\right)  $ denote the vector space of all smooth functions
$\phi$ defined in $\mathbb{R}^{n}\setminus\left\{  \mathbf{a}\right\}  ,$ with
support of the form $K\setminus\left\{  \mathbf{a}\right\}  ,$ where $K$ is
compact in $\mathbb{R}^{n},$ that admit a strong asymptotic expansion of the
form
\begin{equation}
\phi\left(  \mathbf{a}+\mathbf{x}\right)  =\phi\left(  \mathbf{a}%
+r\mathbf{w}\right)  \sim\underset{j=m}{\overset{\infty}{\sum}}a_{j}\left(
\mathbf{w}\right)  r^{j},\ \ \ \text{as }\mathbf{x}\rightarrow\mathbf{0}\,,
\label{Def.1}%
\end{equation}
where $m\in\mathbb{Z}.$ We denote $\mathcal{D}_{\ast,\mathbf{0}}\left(
\mathbb{R}^{n}\right)  $ as $\mathcal{D}_{\ast}\left(  \mathbb{R}^{n}\right)
.$ The space $\mathcal{D}_{\ast,\mathbf{a}}\left(  \mathbb{R}^{n}\right)
$\ has a natural topology that makes it a complete locally convex topological
vector space \cite{YE2}.

\begin{definition}
\label{Definition A}The space of distributions on $\mathbb{R}^{n}$ with a
thick point at $\mathbf{x=a}$ is the dual space of $\mathcal{D}_{\ast
,\mathbf{a}}\left(  \mathbb{R}^{n}\right)  .$ We denote it by $\mathcal{D}%
_{\ast,\mathbf{a}}^{^{\prime}}\left(  \mathbb{R}^{n}\right)  ,$ or just as
$\mathcal{D}_{\ast}^{^{\prime}}\left(  \mathbb{R}^{n}\right)  $ when
$\mathbf{a=0}.$
\end{definition}

In general, we shall denote by $\Pi$ canonical projections, say from $E$ to
$F,$ if they exist but as $\Pi_{E,F}$ when we would like to emphasize the
spaces. In particular we will need the projection operator $\Pi=\Pi
_{\mathcal{D}_{\ast,\mathbf{a}}^{\prime}\left(  \mathbb{R}^{n}\right)
,\mathcal{D}^{\prime}\left(  \mathbb{R}^{n}\right)  }:\mathcal{D}%
_{\ast,\mathbf{a}}^{\prime}\left(  \mathbb{R}^{n}\right)  \rightarrow
\mathcal{D}^{\prime}\left(  \mathbb{R}^{n}\right)  \,,$ dual of the inclusion
$i:\mathcal{D}\left(  \mathbb{R}^{n}\right)  \rightarrow\mathcal{D}%
_{\ast,\mathbf{a}}\left(  \mathbb{R}^{n}\right)  .$ Observe that
$\mathcal{D}\left(  \mathbb{R}^{n}\right)  ,$\ the space of standard test
functions, is a closed subspace of $\mathcal{D}_{\ast,\mathbf{a}}\left(
\mathbb{R}^{n}\right)  .$

Typical elements of $\mathcal{D}_{\ast,\mathbf{a}}^{\prime}\left(
\mathbb{R}^{n}\right)  $ are the finite part regularizations considered in
Definition \ref{Definition Finite Part} and the thick delta functions of order
$q,$ $g\left(  \mathbf{w}\right)  \delta_{\ast}^{\left[  q\right]  }\left(
\mathbf{x-a}\right)  $ for $g\in\mathcal{D}^{\prime}\left(  \mathbb{S}\right)
$ given as
\begin{equation}
\left\langle g\left(  \mathbf{w}\right)  \delta_{\ast}^{\left[  q\right]
}\left(  \mathbf{x-a}\right)  ,\phi\right\rangle =\frac{1}{C}\left\langle
g\left(  \mathbf{w}\right)  ,a_{q}\left(  \mathbf{w}\right)  \right\rangle \,,
\label{thick}%
\end{equation}
if $\phi\in\mathcal{D}_{\ast,\mathbf{a}}\left(  \mathbb{R}^{n}\right)  $ has
the development (\ref{Def.1}). When $g=1$ they are called plain thick delta functions.

We refer to \cite{YE2} for the definition of the basic operations on thick
distributions, like derivatives, changes of variables, and multiplication by
smooth functions. In general ordinary derivatives are denoted as $\nabla_{i},$
distributional derivatives are denoted as $\overline{\nabla}_{i},$ following
\cite{Farassat}, while thick distributional derivatives are denoted as
$\nabla_{i}^{\ast}.$

\subsubsection{Other spaces of thick
distributions\label{Subsection: Other spaces of thick distribution}}

Let $\mathcal{A}\left(  \mathbb{R}^{n}\right)  $ be a space of test functions
in $\mathbb{R}^{n}$ and let $\mathcal{A}^{\prime}\left(  \mathbb{R}%
^{n}\right)  $ be the corresponding space of distributions\footnote{In the
sense of Zemanian \cite{zem}; we assume that $\mathcal{D}\left(
\mathbb{R}^{n}\right)  \subset\mathcal{A}\left(  \mathbb{R}^{n}\right)
\subset\mathcal{E}\left(  \mathbb{R}^{n}\right)  $ densely and continuously
and that differentiation is a continuous map of $\mathcal{A}\left(
\mathbb{R}^{n}\right)  .$}. Our aim in this section is to construct the spaces
of thick test functions and distributions, $\mathcal{A}_{\ast,\mathbf{a}%
}\left(  \mathbb{R}^{n}\right)  $ and $\mathcal{A}_{\ast,\mathbf{a}}^{\prime
}\left(  \mathbb{R}^{n}\right)  .$ Our construction will apply in multiple
cases. For instance, $\mathcal{A}\left(  \mathbb{R}^{n}\right)  $ can be
$\mathcal{E}\left(  \mathbb{R}^{n}\right)  ,$ the space of all smooth
functions and thus $\mathcal{A}^{\prime}\left(  \mathbb{R}^{n}\right)  $
becomes $\mathcal{E}^{\prime}\left(  \mathbb{R}^{n}\right)  ,$ the space of
distributions with compact support; or $\mathcal{A}\left(  \mathbb{R}%
^{n}\right)  $ can be $\mathcal{S}\left(  \mathbb{R}^{n}\right)  ,$ so that
$\mathcal{A}^{\prime}\left(  \mathbb{R}^{n}\right)  $ becomes the space of
tempered distributions $\mathcal{S}^{\prime}\left(  \mathbb{R}^{n}\right)  .$
The case of $\mathcal{K}\left(  \mathbb{R}^{n}\right)  $ and $\mathcal{K}%
^{\prime}\left(  \mathbb{R}^{n}\right)  $ played a central role in the
asymptotic analysis of thick distributions \cite{YE5}.

\begin{definition}
\label{Def OS.1}Let $\mathcal{A}\left(  \mathbb{R}^{n}\right)  $ be a space of
test functions in $\mathbb{R}^{n}.$ The space $\mathcal{A}_{\ast,\mathbf{a}%
}\left(  \mathbb{R}^{n}\right)  $ consists of those functions $\phi$ defined
in $\mathbb{R}^{n}\setminus\left\{  \mathbf{a}\right\}  $ that can be written
as $\phi_{1}+\phi_{2},$ where $\phi_{1}\in\mathcal{D}_{\ast,\mathbf{a}}\left(
\mathbb{R}^{n}\right)  $ and where $\phi_{2}\in\mathcal{A}\left(
\mathbb{R}^{n}\right)  .$ The topology of $\mathcal{A}_{\ast,\mathbf{a}%
}\left(  \mathbb{R}^{n}\right)  $ is the finest topology induced by the map
$A:\mathcal{D}_{\ast,\mathbf{a}}\left(  \mathbb{R}^{n}\right)  \times
\mathcal{A}\left(  \mathbb{R}^{n}\right)  \rightarrow\mathcal{A}%
_{\ast,\mathbf{a}}\left(  \mathbb{R}^{n}\right)  ,$ $A\left(  \phi_{1}%
,\phi_{2}\right)  =\phi_{1}+\phi_{2}.$ The space of thick distributions
$\mathcal{A}_{\ast,\mathbf{a}}^{\prime}\left(  \mathbb{R}^{n}\right)  $ is the
corresponding dual space.
\end{definition}

The topology of $\mathcal{A}_{\ast,\mathbf{a}}\left(  \mathbb{R}^{n}\right)  $
can actually be described in several ways. Suppose for instance that $\rho
\in\mathcal{D}\left(  \mathbb{R}^{n}\right)  $ is a test function that
satisfies that $\rho\left(  \mathbf{x}\right)  =1$ in a neighborhood of
$\mathbf{x}=\mathbf{a}.$ If $\left\Vert \ \ \ \right\Vert _{1}$ is a
continuous seminorm of $\mathcal{D}_{\ast,\mathbf{a}}\left(  \mathbb{R}%
^{n}\right)  $ while $\left\Vert \ \ \ \right\Vert _{2}$ is a continuous
seminorm of $\mathcal{A}\left(  \mathbb{R}^{n}\right)  ,$ then $\left\Vert
\phi\right\Vert =\max\left\{  \left\Vert \rho\phi\right\Vert _{1},\left\Vert
\left(  1-\rho\right)  \phi\right\Vert _{2}\right\}  $ is a continuous
seminorm of $\mathcal{A}_{\ast,\mathbf{a}}\left(  \mathbb{R}^{n}\right)  $ and
the collection of seminorms so constructed form a basis for the continuous
seminorms of $\mathcal{A}_{\ast,\mathbf{a}}\left(  \mathbb{R}^{n}\right)  .$
The elements of $\mathcal{A}_{\ast,\mathbf{a}}\left(  \mathbb{R}^{n}\right)  $
can be described as those smooth functions defined in $\mathbb{R}^{n}%
\setminus\left\{  \mathbf{a}\right\}  $ that show the behavior of thick test
functions near $\mathbf{x}=\mathbf{a}$ while at infinity show the behavior of
the elements of $\mathcal{A}\left(  \mathbb{R}^{n}\right)  .$ Similar
considerations apply to the dual spaces.

Naturally one may consider spaces of thick test functions and thick
distributions on smooth manifolds. In particular, considering the one point
compactification $\mathbb{R}_{\text{c}}^{n}=\mathbb{R}^{n}\cup\left\{
\mathbf{\infty}\right\}  ,$ that can be identified with a sphere in dimension
$n+1,$ we obtain $\mathcal{D}_{\ast,\mathbf{\infty}}\left(  \mathbb{R}%
_{\text{c}}^{n}\right)  ,$\ the space of smooth functions in $\mathbb{R}^{n}$
with a thick point at $\mathbf{\infty},$ namely, smooth functions $\phi$ such
that $\psi\left(  \mathbf{x}\right)  =\phi\left(  \mathbf{x/}\left\vert
\mathbf{x}\right\vert ^{2}\right)  $ has a thick point at the origin. We can
also consider another simple modification of thick test functions, namely, by
considering functions whose expansion at the thick point is given not in terms
of the asymptotic sequence $\left\{  r^{j}\right\}  $ but in terms of another
asymptotic sequence. The topology of such spaces can be constructed in a
completely analogous fashion. In this article we will need to consider test
functions with expansions in terms of the sequence $\left\{  r^{j}\ln
r,r^{j}\right\}  ,$ the space $\mathcal{W}_{\mathrm{pre}}\left(
\mathbb{R}^{n}\right)  $ of the Definition \ref{Definition b 1}.

\subsection{Finite parts\label{Subsection:Finite parts}}

Let us now recall the notion of the finite part of a limit \cite[Section
2.4]{GreenBook}. Let $X$ be a topological space, and let $x_{0}\in X.$ Suppose
$\mathfrak{F},$ the basic functions, is a family of strictly positive
functions defined for $x\in V\setminus\left\{  x_{0}\right\}  ,$ where $V$ is
a neighborhood of $x_{0},$ such that all of them tend to infinity at $x_{0}$
and such that, given two different elements $f_{1},f_{2}\in\mathfrak{F}$, then
$\lim_{x\rightarrow x_{0}}f_{1}\left(  x\right)  /f_{2}\left(  x\right)  $ is
either $0$ or $\infty.$

\begin{definition}
\label{Def FP}Let $G\left(  x\right)  $ be a function defined for $x\in
V\setminus\left\{  x_{0}\right\}  $ with $\lim_{x\rightarrow x_{0}}G\left(
x\right)  =\infty.$ The finite part of the limit of $G\left(  x\right)  $ as
$x\rightarrow x_{0}$ with respect to $\mathfrak{F}$ exists and equals $A$ if
we can write\footnote{Such a decomposition, if it exists, is \emph{unique}
since any finite number of elements of $\mathfrak{F}$ has to be linearly
independent.} $G\left(  x\right)  =G_{1}\left(  x\right)  +G_{2}\left(
x\right)  ,$ where $G_{1},$ the \emph{infinite part,} is a linear combination
of the basic functions and where $G_{2},$\ the \emph{finite part,} has the
property that the limit $A=\lim_{x\rightarrow x_{0}}G_{2}\left(  x\right)  $
exists. We then employ the notation $\mathrm{F.p.}_{\mathfrak{F}}%
\lim_{x\rightarrow x_{0}}G\left(  \varepsilon\right)  =A\,.$ The Hadamard
finite part limit corresponds to the case when $x_{0}=0$ and $\mathfrak{F}$ is
the family of functions $x^{-\alpha}\left\vert \ln x\right\vert ^{\beta},$
where $\alpha>0$ and $\beta\geq0$ or where $\alpha=0$ and $\beta>0,$ or when
$x_{0}=\infty$ and $\mathfrak{F}$ is the family of functions $x^{\alpha
}\left\vert \ln x\right\vert ^{\beta},$ where $\alpha>0$ and $\beta\geq0$ or
where $\alpha=0$ and $\beta>0.$ We then use the simpler notations
$\mathrm{F.p.}\lim_{x\rightarrow0^{+}}G\left(  x\right)  $ or $\mathrm{F.p.}%
\lim_{x\rightarrow\infty}G\left(  x\right)  .$
\end{definition}

Consider now a function $f$ defined in $\mathbb{R}^{n}$ that may or may not be
integrable over the whole space but which is integrable in the region
$\left\vert \mathbf{x}\right\vert >\varepsilon$ for any $\varepsilon>0.$ Then
the \emph{radial} finite part integral is defined as%
\begin{equation*}
\mathrm{F.p.}\int_{\mathbb{R}^{n}}f\left(  \mathbf{x}\right)  \,\mathrm{d}%
\mathbf{x}=\mathrm{F.p.}\lim_{\varepsilon\rightarrow0^{+}}\int_{\left\vert
\mathbf{x}\right\vert >\varepsilon}f\left(  \mathbf{x}\right)  \,\mathrm{d}%
\mathbf{x}\,, 
\end{equation*}
if the finite part limit exists. The notion of finite part integrals and its
name were introduced by Hadamard \cite{Hadamard}, who used them in his study of
fundamental solutions of partial differential equations.

\begin{definition}
\label{Definition Finite Part}If $g$ is a locally integrable function in
$\mathbb{R}^{n}\setminus\left\{  \mathbf{0}\right\}  $ such that the radial
finite part integral of $g\phi$ exists for each $\phi$ belonging to a space of
thick test functions $\mathcal{A}_{\ast}\left(  \mathbb{R}^{n}\right)  ,$ then
we can define a thick\footnote{The same notation is employed for standard
distributions.} distribution $\mathcal{P}f\left\{  g\left(  \mathbf{z}\right)
;\mathbf{x}\right\}  =\mathcal{P}f\left(  g\right)  \in\mathcal{A}_{\ast
}^{\prime}\left(  \mathbb{R}^{n}\right)  $ as%
\begin{equation*}
\left\langle \mathcal{P}f\left\{  g\left(  \mathbf{z}\right)  ;\mathbf{x}%
\right\}  ,\phi\left(  \mathbf{x}\right)  \right\rangle =\left\langle
\mathcal{P}f\left(  g\right)  ,\phi\right\rangle =\mathrm{F.p.}\int
_{\mathbb{R}^{n}}g\left(  \mathbf{x}\right)  \phi\left(  \mathbf{x}\right)
\,\mathrm{d}\mathbf{x}\,. 
\end{equation*}

\end{definition}

The notation $\mathcal{P}f\left(  f\left(  \mathbf{x}\right)  \right)  $ was
introduced by Schwartz \cite[Chp. 2, \S 2]{Schwartz}, who called it a
pseudofunction, a term that is still in use.

A particularly important case of finite part limits is the finite part of a
meromorphic $f$\ function at a pole $\omega,$ which is exactly the value of
the regular part of $f,$ say $g,$ at the pole: $\mathrm{F.p.}\lim
_{\lambda\rightarrow\omega}f\left(  \lambda\right)  =g\left(  \omega\right)
.$

\begin{example}
\label{example 1}Let $x_{0}>0$ and let $\varphi$ be a continuous function in
$[x_{0},\infty),$ that satisfies the asymptotic relation $\varphi\left(
x\right)  =Ax^{\beta}+Bx^{\beta}\ln x+o\left(  x^{-\infty}\right)  $ as
$x\rightarrow\infty.$ The finite part integral $F\left(  \lambda\right)
=\mathrm{F.p.}\int_{x_{0}}^{\infty}x^{\lambda}\varphi\left(  x\right)
\,\mathrm{d}x$ exists for all $\lambda\in\mathbb{C},$ and $F$ will be a
meromorphic function, with a double pole at $\lambda=-\beta-1,$ with singular
part $B\left(  \lambda+\beta+1\right)  ^{-2}-A\left(  \lambda+\beta+1\right)
^{-1} ,$ and the finite part is
\begin{equation}
\mathrm{F.p.}\lim_{\lambda\rightarrow-\beta-1}F\left(  \lambda\right)
=\mathrm{F.p.}\int_{x_{0}}^{\infty}x^{-\beta-1}\varphi\left(  x\right)
\,\mathrm{d}x\,. \label{FP 4}%
\end{equation}
Notice that we have two very different finite part limits in (\ref{FP 4}) and
in this case they give the same result; in fact, that is usually true with
radial finite part integrals \cite{YE5} but not otherwise \cite{Farassat, YE}.
\end{example}

\section{Some Fourier transforms\label{Section: Some Fourier Transforms}}

We need the Fourier transform of several distributions in $\mathbb{R}^{n}%
$\ for later use, especially transforms of the type $\mathcal{F}\left\{
\mathcal{P}f\left(  r^{-N}\right)  a\left(  \mathbf{w}\right)  ;\mathbf{u}%
\right\}  $ where $\mathbf{x}=r\mathbf{w}$ are polar coordinates, and where
$a$ is a smooth function defined on the unit sphere $\mathbb{S}.$ The formulas
for such transforms are available \cite{Samko} but we preferred to present a
self consistent approach to their derivation since these ideas will be useful
when considering the Fourier transform of thick distributions.

We start with the case when $a=1.$ In this case it is well known that%
\begin{equation}
\mathcal{F}\left\{  r^{\lambda};\mathbf{u}\right\}  =\frac{\pi^{n/2}%
2^{\lambda+n}\Gamma\left(  \frac{\lambda+n}{2}\right)  s^{-\lambda-n}}%
{\Gamma\left(  -\frac{\lambda}{2}\right)  \,}\,, \label{2}%
\end{equation}
whenever $\lambda\neq-n,-n-2,-n-4,\ldots$ \cite{GreenBook, Horvath, Kanwal,
Schwartz}. Here $\mathbf{u}=s\mathbf{v}$ are polar coordinates. Observe that
$\mathcal{F}\left\{  r^{\lambda};\mathbf{u}\right\}  $ is in fact analytic at
$\lambda=0,2,4,\ldots,$ so that the right side can be computed as a limit
--employing (\ref{4}) --, namely,%
\begin{equation*}
\mathcal{F}\left\{  r^{2q};\mathbf{u}\right\}  =\lim_{\lambda\rightarrow
2q}\frac{\pi^{n/2}2^{\lambda+n}\Gamma\left(  \frac{\lambda+n}{2}\right)
s^{-\lambda-n}}{\Gamma\left(  -\frac{\lambda}{2}\right)  }=\left(
2\pi\right)  ^{n}\left(  -1\right)  ^{q}\nabla^{2q}\delta\left(
\mathbf{u}\right)  \,. 
\end{equation*}
This is of course the result we would obtain if we use that $\mathcal{F}%
\left\{  1;\mathbf{u}\right\}  =\left(  2\pi\right)  ^{n}\delta\left(
\mathbf{u}\right)  $ and that $\mathcal{F}\left\{  r^{2q}f\left(
\mathbf{x}\right)  ;\mathbf{u}\right\}  =\left(  -1\right)  ^{q}\nabla
^{2q}\mathcal{F}\left\{  f\left(  \mathbf{x}\right)  ;\mathbf{u}\right\}  .$
Next, let us now find $\mathcal{F}\left\{  \mathcal{P}f\left(  r^{-n-2m}%
\right)  ;\mathbf{u}\right\}  $ for $m=0,1,2,\ldots.$ We have%
\begin{equation}
r^{\lambda}=\frac{c_{m,n}\nabla^{2m}\delta\left(  \mathbf{x}\right)  }{\left(
2m\right)  !\left(  \lambda+2m+n\right)  }+\mathcal{P}f\left(  \frac
{1}{r^{n+2m}}\right)  +O\left(  \lambda+2m+n\right)  \,, \label{4}%
\end{equation}
as $\lambda\rightarrow-\left(  2m+n\right)  ,$ so that we obtain the finite
part limit $\mathrm{F.p.}\lim_{\lambda\rightarrow-\left(  2m+n\right)
}r^{\lambda}=\mathcal{P}f\left(  r^{-n-2m}\right)  .$ Therefore $\mathcal{F}%
\left\{  \mathcal{P}f\left(  r^{-n-2m}\right)  ;\mathbf{u}\right\}  $ equals%
\begin{equation*}
\mathrm{F.p.}\lim_{\lambda\rightarrow-\left(  2m+n\right)  }\mathcal{F}%
\left\{  r^{\lambda};\mathbf{u}\right\}  =\mathrm{F.p.}\lim_{\lambda
\rightarrow-\left(  2m+n\right)  }\frac{\pi^{n/2}2^{\lambda+n}\Gamma\left(
\frac{\lambda+n}{2}\right)  s^{-\lambda-n}}{\Gamma\left(  -\frac{\lambda}%
{2}\right)  \,}\,. 
\end{equation*}
This finite part limit is actually already computed in the first edition of
\cite{Schwartz}. It follows easily from the following lemma \cite{Lebedev,
YE5}.

\begin{lemma}
\label{Lemma 2}Let $k\in\mathbb{N}.$ We have that as $\lambda\rightarrow-k,$%
\begin{equation*}
\Gamma\left(  \lambda\right)  =\frac{\left(  -1\right)  ^{k}}{k!\left(
\lambda+k\right)  }+\frac{\left(  -1\right)  ^{k}\psi\left(  k+1\right)  }%
{k!}+O\left(  \lambda+k\right)  \,, 
\end{equation*}
where $\psi\left(  \lambda\right)  =\Gamma^{\prime}\left(  \lambda\right)
/\Gamma\left(  \lambda\right)  $ is the digamma function so that $\psi\left(
k+1\right)  =\sum_{j=1}^{k}1/j\,-\gamma,$ $\gamma$ being Euler's constant. If
$k=0,1,2,\ldots,$ and $f$ is analytic in a neighborhood of $-k,$%
\begin{equation*}
\mathrm{F.p.}\lim_{\lambda\rightarrow-k}\Gamma\left(  \lambda\right)  f\left(
\lambda\right)  =\frac{\left(  -1\right)  ^{k}\psi\left(  k+1\right)  }%
{k!}f\left(  -k\right)  +\frac{\left(  -1\right)  ^{k}}{k!}f^{\prime}\left(
-k\right)  \,. 
\end{equation*}

\end{lemma}

The ensuing result is therefore obtained.

\begin{lemma}
\label{lemma 2.5}If $m=0,1,2,\ldots$ then%
\[
  \ \mathcal{F}\left\{  \mathcal{P}f\left(  \frac{1}{r^{n+2m}}\right)
;\mathbf{u}\right\} =
  \frac{\left(  -1\right)  ^{m}\pi^{n/2}}{m!\Gamma\left(  \frac{n}%
{2}+m\right)  }\left(  \frac{s}{2}\right)  ^{2m}\left\{  \psi\left(
m+1\right)  +\psi\left(  \frac{n}{2}+m\right)  -2\ln\left(  \frac{s}%
{2}\right)  \right\} 
\]

\end{lemma}

Our next task is to find the Fourier transform of distributions of the form
$\mathcal{P}f\left(  r^{-N}\right)  a\left(  \mathbf{w}\right)  $ when
$a=\mathsf{Y}_{k}$ is a spherical harmonic\footnote{We denote this as
$\mathsf{Y}_{k}\in\mathcal{H}_{k}.$ For more on spherical harmonics, see
\cite{AxlerBourdonRamey, RubinBook}.} of degree $k.$

\begin{lemma}
\label{lemma 3}If $\mathsf{Y}_{k}\in\mathcal{H}_{k}$ and $\lambda
\neq-n-k,-n-k-2,-n-k-4,\ldots$
\begin{equation}
\mathcal{F}\left\{  r^{\lambda}\mathsf{Y}_{k}\left(  \mathbf{w}\right)
;s\mathbf{v}\right\}  =\frac{i^{k}\pi^{n/2}2^{\lambda+n}\Gamma\left(
\frac{k+n+\lambda}{2}\right)  }{\Gamma\left(  \frac{k-\lambda}{2}\right)
}s^{-\left(  \lambda+n\right)  }\mathsf{Y}_{k}\left(  \mathbf{v}\right)  \,.
\label{a7}%
\end{equation}

\end{lemma}

\begin{proof}
Notice that for a general $k,$ the Fourier transform of $r^{\lambda}%
\mathsf{Y}_{k}\left(  \mathbf{w}\right)  ,$ a homogeneous distribution of
degree $\lambda,$ is homogeneous of degree $-\left(  \lambda+n\right)  ,$ so
that the Funk-Hecke formula \cite{Funk, Hecke} as presented in
\cite{Estrada18} yields that%
\begin{equation*}
\mathcal{F}\left\{  r^{\lambda}\mathsf{Y}_{k}\left(  \mathbf{w}\right)
;s\mathbf{v}\right\}  =C_{k,\lambda}s^{-\left(  \lambda+n\right)  }%
\mathsf{Y}_{k}\left(  \mathbf{v}\right)  \,, 
\end{equation*}
for some constants $C_{k,\lambda}$ that depend on $k$ and $\lambda$ but not
otherwise on $\mathsf{Y}_{k}.$ Thus, it suffices to show that for each $k$
(\ref{a7}) holds for just \emph{one} spherical harmonic of degree $k.$ We use
induction on $k.$ If $k=0$ then (\ref{a7}) is exactly (\ref{2}). Let us assume
it true for $k$ and let us prove it for $k+1.$ Indeed, we take $\mathsf{Y}%
_{k+1}\left(  \mathbf{x}\right)  =\mathsf{Y}_{k}\left(  \widetilde{\mathbf{x}%
}\right)  x_{n}$ where $\mathbf{x}=\left(  \widetilde{\mathbf{x}}%
,x_{n}\right)  ,$ so that%
\begin{align*}
\mathcal{F}\left\{  r^{\lambda}\mathsf{Y}_{k+1}\left(  \mathbf{w}\right)
;s\mathbf{v}\right\}  &=-i\frac{\partial}{\partial u_{n}}\left(  \frac{i^{k}%
\pi^{n/2}2^{\lambda-1+n}\Gamma\left(  \frac{k+n+\lambda-1}{2}\right)  }%
{\Gamma\left(  \frac{k-\lambda+1}{2}\right)  \,}s^{1-\lambda-n}\mathsf{Y}%
_{k}\left(  \mathbf{v}\right)  \right)
\\
&
=\frac{-i^{k+1}\pi^{n/2}2^{\lambda-1+n}\Gamma\left(  \frac{k+n+\lambda-1}%
{2}\right)  }{\Gamma\left(  \frac{k-\lambda+1}{2}\right)  }\frac{\partial
}{\partial u_{n}}\left(  \mathsf{Y}_{k}\left(  \mathbf{u}\right)
s^{-\lambda+1-k-n}\right)
\\
&
=\frac{i^{k+1}\pi^{n/2}2^{\lambda+n}\Gamma\left(  \frac{k+1+n+\lambda}%
{2}\right)  }{\Gamma\left(  \frac{k+1-\lambda}{2}\right)  \,}s^{-\lambda
-n}\mathsf{Y}_{k+1}\left(  \mathbf{v}\right)  \,,
\end{align*}
as required.
\end{proof}

Since $\mathcal{F}\left\{  r^{\lambda}\mathsf{Y}_{k}\left(  \mathbf{w}\right)
;s\mathbf{v}\right\}  $ is analytic at $\lambda=k+2q,$ $q=0,1,2,\ldots,$ at
this value of $\lambda$ (\ref{a7}) is the limit of the expression as
$\lambda\rightarrow k+2q.$ In fact, using the product formula
\begin{equation*}
\mathsf{Y}_{k}\left(  \mathbf{u}\right)  \nabla^{2m}\delta\left(
\mathbf{u}\right)  =\frac{\left(  -1\right)  ^{k}2^{k}m!}{\left(  m-k\right)
!}\mathsf{Y}_{k}\left(  \mathbf{\nabla}\right)  \nabla^{2m-2k}\delta\left(
\mathbf{u}\right)  \,,\ \ \ m\geq k\,, %
\end{equation*}
from \cite[Prop. 3.3]{Estrada 18b}, we indeed obtain%
\begin{align}
\mathcal{F}\left\{  r^{k+2q}\mathsf{Y}_{k}\left(  \mathbf{w}\right)
;s\mathbf{v}\right\}   &  =\lim_{\lambda\rightarrow k+2q}\frac{i^{k}\pi
^{n/2}2^{\lambda+n}\Gamma\left(  \frac{k+n+\lambda}{2}\right)  }{\Gamma\left(
\frac{k-\lambda}{2}\right)  \,s^{\lambda+n}}\mathsf{Y}_{k}\left(
\mathbf{v}\right) \label{a.8}\\
&  =\frac{\left(  -i\right)  ^{k}\left(  -1\right)  ^{q}\left(  2\pi\right)
^{n}\left(  k+q\right)  !}{k!q!}\mathsf{Y}_{k}\left(  \mathbf{\nabla}\right)
\nabla^{2q}\delta\left(  \mathbf{u}\right)  \,.\nonumber
\end{align}

If we now use the Lemma \ref{Lemma 2} as before, we obtain the following formula.

\begin{lemma}
\label{Lemma 4}If $\mathsf{Y}_{k}\in\mathcal{H}_{k}$ and $m=0,1,2,\ldots$ then%
\begin{equation*}
\mathcal{F}\left\{  \mathcal{P}f\left(  \frac{1}{r^{n+k+2m}}\right)
\mathsf{Y}_{k}\left(  \mathbf{w}\right)  ;s\mathbf{v}\right\}  = 
\end{equation*}%
\[
\frac{\left(  -1\right)  ^{m}i^{k}\pi^{n/2}}{m!\Gamma\left(  \frac{n}%
{2}+k+m\right)  }\left(  \frac{s}{2}\right)  ^{2m+k}\left\{  \psi\left(
1+m\right)  +\psi\left(  \frac{n}{2}+k+m\right)  -2\ln\left(  \frac{s}%
{2}\right)  \right\}  \mathsf{Y}_{k}\left(  \mathbf{v}\right)  \,.
\]

\end{lemma}

Let now $a$ be a smooth function on the sphere, $a\in\mathcal{D}\left(
\mathbb{S}\right)  .$ Then we can write it in terms of spherical harmonics as%
\begin{equation*}
a\left(  \mathbf{w}\right)  =\sum_{m=0}^{\infty}\mathsf{Y}_{m}\left(
\mathbf{w}\right)  \,, 
\end{equation*}
where $\mathsf{Y}_{m}=\mathsf{Y}_{m}\left\{  a\right\}  \in\mathcal{H}_{m}$
are given as $\mathsf{Y}_{m}\left(  \mathbf{w}\right)  =\int_{\mathbb{S}%
}\mathsf{Z}_{m}\left(  \mathbf{w},\mathbf{v}\right)  a\left(  \mathbf{v}%
\right)  \,\mathrm{d}\sigma\left(  \mathbf{v}\right)  ;$ here $\mathsf{Z}%
_{m}\left(  \mathbf{w},\mathbf{v}\right)  $\ is the reproducing kernel of
$\mathcal{H}_{m},$ namely \cite[Thm. 5.38]{AxlerBourdonRamey}
\begin{equation*}
\left(  n+2m-2\right)  \sum_{q=0}%
^{\hbox{$[\kern-.41em [$}\,\,m/2\,\hbox{$]\kern-.41em ]$}}\left(  -1\right)
^{q}\frac{n\left(  n+2\right)  \cdots\left(  n+2m-2q-4\right)  }%
{2^{q}q!\left(  m-2q\right)  !}\left(  \mathbf{w}\cdot\mathbf{v}\right)
^{m-2q}. 
\end{equation*}
We thus obtain the following.

\begin{proposition}
\label{Lemma 5}If $\beta\neq0,1,2,\ldots$ then%
\begin{equation*}
\mathcal{F}\left\{  \mathcal{P}f\left(  \frac{1}{r^{n+\beta}}\right)  a\left(
\mathbf{w}\right)  ;\mathbf{u}\right\}  =\mathcal{P}f\left(  s^{\beta}\right)
\mathcal{K}_{\beta}\left\{  a\left(  \mathbf{w}\right)  ;\mathbf{v}\right\}
\,,
\end{equation*}
where $\mathcal{K}_{\beta}\left\{  a\left(  \mathbf{w}\right)  ;\mathbf{v}%
\right\}  =\left\langle K_{\beta}\left(  \mathbf{w},\mathbf{v}\right)
,a\left(  \mathbf{w}\right)  \right\rangle _{\mathbf{w}}\,,$and%
\begin{equation}
K_{\beta}\left(  \mathbf{w},\mathbf{v}\right)  =\sum_{m=0}^{\infty}%
\kappa_{\beta,m}\mathsf{Z}_{m}\left(  \mathbf{w},\mathbf{v}\right)
\,,\ \ \kappa_{\beta,m}=\frac{i^{m}\pi^{n/2}2^{-\beta}\Gamma\left(
\frac{m-\beta}{2}\right)  }{\Gamma\left(  \frac{m+n+\beta}{2}\right)  }\,.
\label{a.13p}%
\end{equation}

\end{proposition}

The operator $\mathcal{K}_{\beta}$ is analytic for $\beta\neq0,1,2,\ldots$;
for $\beta=q\in\mathbb{N}$ we have the next formula.

\begin{proposition}
\label{Lemma 6}If $q=0,1,2,\ldots$ then
\begin{equation*}
\mathcal{F}\left\{  \mathcal{P}f\left(  \frac{1}{r^{n+q}}\right)  a\left(
\mathbf{w}\right)  ;\mathbf{u}\right\}  =s^{q}\left(  \mathcal{K}_{q}\left\{
a\left(  \mathbf{w}\right)  ;\mathbf{v}\right\}  +\mathcal{L}_{q}\left\{
a\left(  \mathbf{w}\right)  ;\mathbf{v}\right\}  \ln s\right)  \,,
\end{equation*}
where $\mathcal{K}_{q}\left\{  a\left(  \mathbf{w}\right)  ;\mathbf{v}%
\right\}  =\left\langle K_{q}\left(  \mathbf{w},\mathbf{v}\right)  ,a\left(
\mathbf{w}\right)  \right\rangle _{\mathbf{w}}\,,$%
\begin{equation*}
K_{q}\left(  \mathbf{w},\mathbf{v}\right)  =\sum_{m=0}^{\infty}\kappa
_{q,m}\mathsf{Z}_{m}\left(  \mathbf{w},\mathbf{v}\right)  \,, 
\end{equation*}
the constants $\kappa_{q,m}$ being given by (\ref{a.13p}) if $m\neq
q,q-2,\ldots$ and as%
\begin{equation*}
\kappa_{q,q-2m}=\frac{i^{q}\pi^{n/2}2^{-q}}{m!\Gamma\left(  \frac{n}%
{2}+q-m\right)  }\left\{  \psi\left(  1+m\right)  +\psi\left(  \frac{n}%
{2}+q-m\right)  +2\ln2\right\}  \,, 
\end{equation*}
for $0\leq m\leq\hbox{$[\kern-.41em [$}\,\,q/2\,\hbox{$]\kern-.41em ]$}.$ On
the other hand, $\mathcal{L}_{q}\left\{  a\left(  \mathbf{w}\right)
;\mathbf{v}\right\}  =\left\langle L_{q}\left(  \mathbf{w},\mathbf{v}\right)
,a\left(  \mathbf{w}\right)  \right\rangle _{\mathbf{w}}\,,$%
\begin{equation}
L_{q}\left(  \mathbf{w},\mathbf{v}\right)  =\sum_{m=0}%
^{\hbox{$[\kern-.41em [$}\,\,q/2\,\hbox{$]\kern-.41em ]$}}\lambda
_{q,q-2m}\mathsf{Z}_{q-2m}\left(  \mathbf{w},\mathbf{v}\right)
\,,\ \ \ \lambda_{q,q-2m}=\frac{-i^{q}2^{-q+1}\pi^{n/2}}{m!\Gamma\left(
\frac{n}{2}+q-m\right)  }\,. \label{a.16p}%
\end{equation}

\end{proposition}

\subsection{The operators $\mathcal{K}_{\beta}$%
\label{Subsection: The operators}}

Notice that $\mathcal{K}_{\beta}$ is the analytic continuation of an integral
operator $a\rightsquigarrow\int_{\mathbb{S}}K_{\beta}\left(  \mathbf{w}%
,\mathbf{v}\right)  a\left(  \mathbf{w}\right)  \,\mathrm{d}\sigma\left(
\mathbf{w}\right)  ,$ namely, if $\Re e\,\beta>0$ then employing polar
coordinates we obtain%
\begin{equation*}
\mathcal{F}\left\{  r^{-n-\beta}a\left(  \mathbf{w}\right)  ;\mathbf{u}%
\right\}  =s^{\beta}\Gamma\left(  -\beta\right)  e^{-i\pi\beta/2}%
\int_{\mathbb{S}}a\left(  \mathbf{w}\right)  \left(  \mathbf{w\cdot
v}+i0\right)  ^{\beta}\,\mathrm{d}\sigma\left(  \mathbf{w}\right)  \,,
\end{equation*}
since in dimension $1$ \cite{Kanwal} $\mathcal{F}\left\{  x_{+}^{-1-\beta
};t\right\}  =\Gamma\left(  -\beta\right)  e^{-i\pi\beta/2}\left(
t+i0\right)  ^{\beta}$ if $\beta\neq0,1,2,\ldots$ Thus,
\begin{equation}
K_{\beta}\left(  \mathbf{w},\mathbf{v}\right)  =\Gamma\left(  -\beta\right)
e^{-i\pi\beta/2}\left(  \mathbf{w\cdot v}+i0\right)  ^{\beta}, \label{a.19}%
\end{equation}
a distributional kernel for $\beta\neq0,1,2,\ldots$ that becomes an integral
operator if $\Re e\,\beta>0.$

Observe that the distribution $\left(  t+i0\right)  ^{\beta}$ is an entire
function of $\beta.$ The singularity of $K_{\beta}\left(  \mathbf{w}%
,\mathbf{v}\right)  $ at $\beta=q\in\mathbb{N}$ is produced by the term
$\Gamma\left(  -\beta\right)  .$ The formula of the Proposition \ref{Lemma 6}
can therefore be derived by computing the finite part of the limit of
$K_{\beta}\left(  \mathbf{w},\mathbf{v}\right)  s^{\beta}$\ as $\beta
\rightarrow q,$ because the Lemma \ref{Lemma 2} gives%
\begin{equation*}
\mathrm{F.p.}\lim_{\lambda\rightarrow-q}\Gamma\left(  \lambda\right)
a^{\lambda}=\frac{\left(  -1\right)  ^{q}}{q!}\left(  \psi\left(  q+1\right)
+\ln a\right)  a^{-q}. 
\end{equation*}
Hence, since $\left(  \mathbf{w\cdot v}+i0\right)  ^{q}=\left(  \mathbf{w\cdot
v}\right)  ^{q},$
\begin{equation}
K_{q}\left(  \mathbf{w},\mathbf{v}\right)  =\frac{\left(  -1\right)
^{q}e^{-i\pi q/2}}{q!}\left(  \psi\left(  q+1\right)  +\ln\left(
\frac{e^{i\pi/2}}{\left(  \mathbf{w\cdot v}+i0\right)  }\right)  \right)
\left(  \mathbf{w\cdot v}\right)  ^{q}, \label{a.21}%
\end{equation}
and%
\begin{equation}
L_{q}\left(  \mathbf{w},\mathbf{v}\right)  =\frac{\left(  -1\right)
^{q}e^{-i\pi q/2}}{q!}\left(  \mathbf{w\cdot v}\right)  ^{q}. \label{a.22}%
\end{equation}

It also interesting to observe the form of $K_{-m}\left(  \mathbf{w}%
,\mathbf{v}\right)  $ for $m=1,2,3,\ldots,$%
\begin{align}
K_{-m}\left(  \mathbf{w},\mathbf{v}\right) & =\left(  m-1\right)  !i^{m}\left(
\mathbf{w\cdot v}+i0\right)  ^{-m} \label{a.23}%
\\
&  =\left(  m-1\right)  !i^{m}\left(  \mathbf{w\cdot v}\right)  ^{-m}%
-\pi\left(  -i\right)  ^{m+1}\delta^{\left(  m-1\right)  }\left(
\mathbf{w\cdot v}\right) 
\nonumber
\\
&  =-2\pi\left(  -i\right)  ^{m+1}\delta^{+\left(  m-1\right)  }\left(
\mathbf{w\cdot v}\right)  \,, \nonumber
\end{align}
where $\delta^{+\left(  m-1\right)  }\left(  x\right)  $ is the Heisenberg
delta function \cite[(2.61), (2.63)]{GreenBook}.

If $\beta\neq0,1,2,\ldots,$ the coefficients $\kappa_{\beta,m}$ never vanish
for $\beta\neq-n-q,$ $q=0,1,2,\ldots$, but they could vanish for some
$m$\ when $\beta=-n-q,$ so that the operator $\mathcal{K}_{\beta}$ is an
isomorphism of $\mathcal{D}\left(  \mathbb{S}\right)  $\ for $\beta\neq-n-q,$
but $\mathcal{K}_{-n-q}\left(  \mathcal{D}\left(  \mathbb{S}\right)  \right)
$ is a subspace of finite codimension of $\mathcal{D}\left(  \mathbb{S}%
\right)  .$ The Fourier inversion formula yields the inverses of the operators
$\mathcal{K}_{\beta}$ for $\beta\in\mathbb{C}\setminus\mathbb{Z}$ or for
$\beta\in\left\{  1-n,2-n,\ldots,-1\right\}  $ as%
\begin{equation}
\mathcal{K}_{\beta}^{-1}\left\{  A\left(  \mathbf{v}\right)  ;\mathbf{w}%
\right\}  =\frac{1}{\left(  2\pi\right)  ^{n}}\mathcal{K}_{-n-\beta}\left\{
A\left(  \mathbf{v}\right)  ;-\mathbf{w}\right\}  \,,\ \ \ A\in\mathcal{D}%
\left(  \mathbb{S}\right)  \,. \label{a.24}%
\end{equation}

\subsection{The operators $\mathfrak{K}_{q}$ and $\mathfrak{L}_{q}%
$\label{Subsection: the operators k and l}}

It is convenient to consider a variant of the operators $\mathcal{K}_{\beta}$
in case $\beta\in\mathbb{Z}.$ Let us start with some notation. If
$q\in\mathbb{N}$ we denote as $\mathcal{P}_{q}$ the space of restrictions of
homogeneous polynomials of degree $q$ to $\mathbb{S},$ that is $\mathcal{P}%
_{q}=\mathcal{H}_{q}\oplus\mathcal{H}_{q-2}\oplus\mathcal{H}_{q-4}\oplus
\cdots.$\ Let $\mathcal{X}$ be a space of functions or generalized functions
over $\mathbb{S},$ as $\mathcal{D}\left(  \mathbb{S}\right)  ,$ $L^{2}\left(
\mathbb{S}\right)  ,$ or $\mathcal{D}^{\prime}\left(  \mathbb{S}\right)  ,$
that equals the closure in $\mathcal{X}$ of the sum $\mathcal{H}_{0}%
\oplus\mathcal{H}_{1}\oplus\mathcal{H}_{2}\oplus\cdots$\footnote{Such closures
will be denoted as $\widehat{\bigoplus}_{\mathcal{X}}\left.  ~\right\vert
_{m=0}^{\infty}\mathcal{H}_{m}$}$.$ Then $\mathcal{X}_{q}$ is the space
$\mathcal{X}$ if $1-n\leq q\leq-1;$ if $q\geq0,$ $\mathcal{X}_{q}$ is the sum
$\widehat{\bigoplus}_{\mathcal{X}}\left.  ~\right\vert _{m\neq q,q-2.\ldots
}\mathcal{H}_{m},$ while if $q\leq-n$ then $\mathcal{X}_{q}=\mathcal{X}%
_{-n-q}.$ Notice that%
\begin{equation*}
\mathcal{X}_{q}\oplus\mathcal{P}_{-n-q}=\mathcal{X}\,,\ \ q\leq
-n\,,\ \ \ \ \mathcal{X}_{q}\oplus\mathcal{P}_{q}=\mathcal{X}\,,\ \ q\geq0\,.
\end{equation*}

We define the operators $\mathfrak{K}_{q}:\mathcal{D}_{q}\longrightarrow
\mathcal{D}_{q}$ as $\Pi\mathcal{K}_{q}\iota,$\ where $\iota$ is the canonical
injection of $\mathcal{D}_{q}$ into $\mathcal{D}\left(  \mathbb{S}\right)  $
and $\Pi$ the canonical projection of $\mathcal{D}\left(  \mathbb{S}\right)  $
onto $\mathcal{D}_{q}.$ We can also consider the $\mathfrak{K}_{q}$ as
operators from $\mathcal{D}_{q}^{\prime}$ to itself, by duality or employing
the expansion (\ref{a.13p}). The Propositions \ref{Lemma 5} and \ref{Lemma 6}
immediately give the ensuing.

\begin{proposition}
\label{Propposition kq}The operators $\mathfrak{K}_{q}$ are isomorphisms of
the space $\mathcal{X}_{q}$ to itself for\footnote{The results also holds for
$\mathcal{X}=L^{2}\left(  \mathbb{S}\right)  ,$ but we will not need this case
presently.} $\mathcal{X}=\mathcal{D}\left(  \mathbb{S}\right)  $ or
$\mathcal{D}^{\prime}\left(  \mathbb{S}\right)  .$ Its inverses are given as%
\begin{equation}
\mathfrak{K}_{q}^{-1}\left\{  A\left(  \mathbf{v}\right)  ;\mathbf{w}\right\}
=\frac{1}{\left(  2\pi\right)  ^{n}}\mathfrak{K}_{-n-q}\left\{  A\left(
\mathbf{v}\right)  ;-\mathbf{w}\right\}  \,,\ \ \ A\in\mathcal{X}_{q}\,.
\label{a 28}%
\end{equation}

\end{proposition}

Observe that for $\mathcal{X}=\mathcal{D}\left(  \mathbb{S}\right)  $ or
$\mathcal{D}^{\prime}\left(  \mathbb{S}\right)  $ we have $\mathcal{X}%
_{q}=\mathcal{K}_{q}\left(  \mathcal{X}\right)  ,$ for $q<0.$ This is not true
for $q\geq0,$ but we have $\mathcal{X}_{q}=\Pi\mathcal{K}_{q}\left(
\mathcal{X}\right)  $ where $\Pi$ is the canonical projection of $\mathcal{X}$
onto $\mathcal{X}_{q}.$

The operators $\mathfrak{L}_{q}:\mathcal{P}_{q}\longrightarrow\mathcal{P}_{q}$
are defined as $\Pi\mathcal{L}_{q}\iota,$\ where $\iota$ is the canonical
injection of $\mathcal{P}_{q}$ into $\mathcal{D}\left(  \mathbb{S}\right)  $
and $\Pi$ the canonical projection of $\mathcal{D}\left(  \mathbb{S}\right)  $
onto $\mathcal{P}_{q}.$ They are isomorphisms of the space $\mathcal{P}_{q}.$

\section{The Fourier transform of thick test
functions\label{Section: The Fourier transform of thick test functions}}

In this section we will construct a space $\mathcal{W}\left(  \mathbb{R}%
^{n}\right)  $ such that it is possible to define an operator
\begin{equation}
\mathcal{F}_{\ast,\text{t}}:\mathcal{S}_{\ast}\left(  \mathbb{R}^{n}\right)
\longrightarrow\mathcal{W}\left(  \mathbb{R}^{n}\right)  \,, \label{b 1}%
\end{equation}
the Fourier transform of \emph{test functions,} which has the expected
properties of such a transform.

Let us start by observing that if $\phi$ is a thick test function in
$\mathbb{R}^{n},$ then in general it is \emph{not} locally integrable at the
origin, so that, in general, it does not give a unique distribution.
Therefore, we cannot imbed $\mathcal{S}_{\ast}\left(  \mathbb{R}^{n}\right)  $
into $\mathcal{S}^{\prime}\left(  \mathbb{R}^{n}\right)  $ and consequently,
if $\phi\in\mathcal{S}_{\ast}\left(  \mathbb{R}^{n}\right)  $ then in general
we cannot define $\mathcal{F}\left(  \phi\right)  $ as a distribution of the
space $\mathcal{S}^{\prime}\left(  \mathbb{R}^{n}\right)  $\footnote{It is
possible to consider $\mathcal{F}\left(  \phi\right)  $\ as a distribution of
the Lizorkin distributional spaces, but for our purposes a different approach
is more convenient.}. On the other hand, any $\phi\in\mathcal{S}_{\ast}\left(
\mathbb{R}^{n}\right)  $ does have regularizations $f\in\mathcal{S}^{\prime
}\left(  \mathbb{R}^{n}\right)  ;$ however $f$ is not unique, since if $f_{0}$
is a regularization, then so are all distributions of the form $f_{0}+g,$
where $\operatorname*{supp}g\subset\left\{  \mathbf{0}\right\}  ,$ that is,
where $g$ is a sum of derivatives of the Dirac delta function at the origin.
It will be convenient to use the notation $\mathcal{S}_{\ast,\text{reg}%
}\left(  \mathbb{R}^{n}\right)  $ for the subspace of $\mathcal{S}^{\prime
}\left(  \mathbb{R}^{n}\right)  $ whose elements are the regularizations of
thick test functions.

Our first task is then to identify those distributions of the form
$\mathcal{F}\left(  f_{0}\right)  $ where $f_{0}$ is a regularization of a
thick test function $\phi\in\mathcal{S}_{\ast}\left(  \mathbb{R}^{n}\right)
.$ It should be clear that if $\Phi_{0}=\mathcal{F}\left(  f_{0}\right)  $ for
one such regularization of $\phi,$ then so are all distributions of the form
$\Phi_{0}+p$ for any polynomial $p$ and, conversely, if $\Phi$ is the Fourier
transform of a regularization of $\phi$ then $\Phi=\Phi_{0}+p$ for some
polynomial $p.$ Observe now that if $\phi\in\mathcal{S}_{\ast}\left(
\mathbb{R}^{n}\right)  $ then $\phi$ coincides with a test function of the
space $\mathcal{S}\left(  \mathbb{R}^{n}\right)  $ outside any ball around the
origin, while at the origin it has a strong asymptotic expansion of the form
$\phi\left(  \mathbf{x}\right)  \sim\sum_{m=-M}^{\infty}a_{m}\left(
\mathbf{w}\right)  r^{m},$ as $r\rightarrow0^{+},$ where $a_{m}\in
\mathcal{D}\left(  \mathbb{S}\right)  $. We can therefore readily obtain the
properties of the Fourier transform $\Phi_{0}=\mathcal{F}\left(  f_{0}\right)
$ of the pseudofunction $f_{0}=\mathcal{P}f\left(  \phi\right)  ,$ the finite
part regularization of $\phi.$ Indeed, $\Phi_{0}$ is smooth in all of
$\mathbb{R}^{n},$ and our analysis of the Section
\ref{Section: Some Fourier Transforms} combined with the techniques of
\cite{YE5} or of \cite[Chpt. 4]{GreenBook} yield the\ asymptotic expansion of
$\Phi_{0}\left(  \mathbf{u}\right)  $ as $\left\vert \mathbf{u}\right\vert
\rightarrow\infty$ as follows: if $\mathbf{u}=s\mathbf{v}$ are polar
coordinates then we have the strong expansion%
\begin{align}
\Phi_{0}\left(  s\mathbf{v}\right)   &  \sim\sum_{m\leq-n}s^{-m-n}\left(
\mathcal{K}_{-m-n}\left\{  a_{m}\left(  \mathbf{w}\right)  ;\mathbf{v}%
\right\}  +\mathcal{L}_{-m-n}\left\{  a_{m}\left(  \mathbf{w}\right)
;\mathbf{v}\right\}  \ln s\right) \nonumber\\
& \qquad  +\sum_{m=-n+1}^{\infty}s^{-m-n}\mathcal{K}_{-m-n}\left\{  a_{m}\left(
\mathbf{w}\right)  ;\mathbf{v}\right\}  \,, \label{b 3}%
\end{align}
as $s\rightarrow\infty,$ uniformly with respect to $\mathbf{v.}$ Consequently
we introduce the space $\mathcal{W}_{\mathrm{pre}}\left(  \mathbb{R}%
^{n}\right)  .$

\begin{definition}
\label{Definition b 1}The space $\mathcal{W}_{\mathrm{pre}}\left(
\mathbb{R}^{n}\right)  $ consists of those smooth functions $\Phi$\ defined in
$\mathbb{R}^{n}$ that admit a strong asymptotic expansion of the form%
\begin{equation}
\Phi\left(  s\mathbf{v}\right)  \sim\sum_{q=0}^{Q}\left(  A_{q}\left(
\mathbf{v}\right)  +P_{q}\left(  \mathbf{v}\right)  \ln s\right)  s^{q}%
+\sum_{q=1}^{\infty}A_{-q}\left(  \mathbf{v}\right)  s^{-q}, \label{b 4}%
\end{equation}
where $A_{q}\in\mathcal{K}_{q}\left(  \mathcal{D}\left(  \mathbb{S}\right)
\right)  $ for $q\leq-n,$ $A_{q}\in\mathcal{D}\left(  \mathbb{S}\right)  $ for
$q>-n,$ and where the $P_{q}\in\mathcal{P}_{q}$ for $q\in\mathbb{N}.$ The
topology of $\mathcal{W}_{\mathrm{pre}}\left(  \mathbb{R}^{n}\right)  $ is
constructed as explained in Subsection
\ref{Subsection: Other spaces of thick distribution}.
\end{definition}

Our analysis so far yields the ensuing result.

\begin{theorem}
\label{Theorem b 1}The Fourier transform is an isomorphism of the vector
spaces $\mathcal{S}_{\ast,\text{reg}}\left(  \mathbb{R}^{n}\right)  $ and
$\mathcal{W}_{\mathrm{pre}}\left(  \mathbb{R}^{n}\right)  .$
\end{theorem}

Notice that we have not defined a topology for the space $\mathcal{S}%
_{\ast,\text{reg}}\left(  \mathbb{R}^{n}\right)  $ yet; once a topology is
introduced, we shall see that the Fourier transform is not only an algebraic
isomorphism, but actually an isomorphism of topological vector spaces. First,
however, we need to consider the notions of delta part and polynomial part of distributions.

\subsection{Delta parts and polynomial
parts\label{Subsection:Delta parts and polynomial parts}}

In general it is not possible to separate the contribution to a distribution
from a given point; to talk about the \textquotedblleft delta part at
$\mathbf{x}_{0}$\textquotedblright\ of \emph{all} distributions does not make
sense. However, \emph{sometimes,} we can actually separate the delta part
\cite{Estrada16}.

\begin{definition}
\label{Def delta part}Let $f_{0}\in\mathcal{D}^{\prime}\left(  \mathbb{R}%
^{n}\setminus\left\{  \mathbf{0}\right\}  \right)  $ be a distribution defined
in the complement of the origin. Suppose the pseudofunction $\mathcal{P}%
f\left(  f_{0}\left(  \mathbf{x}\right)  \right)  $ exists in $\mathcal{D}%
^{\prime}\left(  \mathbb{R}^{n}\right)  $ (respectively in $\mathcal{D}_{\ast
}^{\prime}\left(  \mathbb{R}^{n}\right)  $). Let $f\in\mathcal{D}^{\prime
}\left(  \mathbb{R}^{n}\right)  $ (respectively in $f\in\mathcal{D}_{\ast
}^{\prime}\left(  \mathbb{R}^{n}\right)  $) be any regularization of $f_{0}.$
Then the delta part at $0$ of $f$ is the distribution $f-\mathcal{P}f\left(
f_{0}\left(  \mathbf{x}\right)  \right)  ,$ whose support is the
origin\footnote{Notice that this delta part is in fact a \emph{spherical}
delta part.}.
\end{definition}

It is easy to construct distributions whose delta part is not defined. Indeed,
the function $\sin r^{-k}$ is locally integrable in $\mathbb{R}^{n},$\ and
thus it gives a well defined regular distribution in $\mathcal{D}^{\prime
}\left(  \mathbb{R}^{n}\right)  .$ If $k>n,$ then the distributional
derivative $\left(  \overline{\nabla}_{i}\right)  \sin r^{-k}$ is another well
defined distribution, but its delta part at the origin is \emph{not} defined,
since $\mathcal{P}f\left(  \nabla_{i}\sin r^{-k}\right)  $ does not exist.

It must be emphasized that even though when it exists $\mathcal{P}f\left(
f_{0}\left(  \mathbf{x}\right)  \right)  $ is in a way \emph{the} natural
regularization of $f_{0},$ other regularizations appear also very naturally,
as we now illustrate. Consider the distribution $\mathcal{P}f\left(
r^{-k}\right)  $ in $\mathbb{R}^{n}.$ Then the distributional derivative
$\overline{\nabla}_{i}\mathcal{P}f\left(  r^{-k}\right)  $ is a regularization
of $-kx_{i}r^{-k-2},$ the ordinary derivative of $r^{-k};$ however
\cite[(3.16)]{EK89} if $k-n=2m$ is an even positive integer, then%
\begin{equation}
\overline{\nabla}_{i}\mathcal{P}f\left(  r^{-k}\right)  =\mathcal{P}f\left(
-kx_{i}r^{-k-2}\right)  -\frac{c_{m,n}}{\left(  2m\right)  !k}\nabla_{i}%
\Delta^{m}\delta\left(  \mathbf{x}\right)  \,, \label{del 1}%
\end{equation}
where $c_{m,n}$ is given by (\ref{ce}). Therefore, $\left(  -c_{m,n}/\left(
2m\right)  !k\right)  \nabla_{i}\Delta^{m}\delta\left(  \mathbf{x}\right)  $
is the delta part of the distribution $\overline{\nabla}_{i}\mathcal{P}%
f\left(  r^{-k}\right)  $ in $\mathcal{D}^{\prime}\left(  \mathbb{R}%
^{n}\right)  .$ In the space $\mathcal{D}_{\ast}^{\prime}\left(
\mathbb{R}^{n}\right)  ,$ now for any integer $k\in\mathbb{Z},$ the delta part
of the thick derivative $\nabla_{i}^{\ast}\mathcal{P}f\left(  r^{-k}\right)  $
is given \cite[Thm.7.1]{YE2} as $Cn_{i}\delta_{\ast}^{\left[  -k-n+1\right]
}.$\footnote{Naturally, when $k-n=2m\geq0,$ the projection of the thick delta
part is precisely the distributional delta part, and this agrees with
\cite[(7.7)]{YE2}.}

Another illustration of \textquotedblleft natural\textquotedblright%
\ regularizations that differ from the finite part is the following. In
$\mathbb{R}^{n}$ for $n\geq2,$ and for $m\in\mathbb{N},$ the distribution
$\lambda^{n+2m}\mathcal{P}f\left(  \left\vert \lambda\mathbf{x}\right\vert
^{-n-2m}\right)  $ is a regularization of $r^{-n-2m}$ and in $\mathcal{D}%
^{\prime}\left(  \mathbb{R}^{n}\right)  $\ its delta part is $\ln
\lambda\,c_{m,n}\mathbf{\nabla}^{2m}\delta\left(  \mathbf{x}\right)  /\left(
2m\right)  !$, while in $\mathcal{D}_{\ast}^{\prime}\left(  \mathbb{R}%
^{n}\right)  $ its delta part is $\ln\lambda\,C\delta_{\ast}^{\left[
2m\right]  },$ as follows from \cite[(5.13), (5.14)]{YE2}. More generally
\cite{Estrada18c} let $A$ be a non singular $n\times n$\ matrix, let
$b\in\mathcal{D}\left(  \mathbb{S}\right)  ,$ and put $B_{\alpha}\left(
\mathbf{z}\right)  =b\left(  \mathbf{z/}\left\vert \mathbf{z}\right\vert
\right)  \left\vert \mathbf{z}\right\vert ^{\alpha},$ the extension to
$\mathcal{D}^{\prime}\left(  \mathbb{R}^{n}\setminus\left\{  \mathbf{0}%
\right\}  \right)  $ that is homogeneous of degree $\alpha.$ Then in
$\mathcal{D}_{\ast}^{\prime}\left(  \mathbb{R}^{n}\right)  $ we have
$\mathcal{P}f\left\{  B_{\alpha}\left(  A\mathbf{z}\right)  ;\mathbf{x}%
\right\}  =\mathcal{P}f\left\{  B_{\alpha}\left(  \mathbf{z}\right)
;A\mathbf{x}\right\}  $ if $\alpha\notin\mathbb{Z},$ but if $\alpha
=k\in\mathbb{Z},$%
\begin{equation}
\mathcal{P}f\left\{  B_{k}\left(  A\mathbf{z}\right)  ;\mathbf{x}\right\}
=\mathcal{P}f\left\{  B_{k}\left(  \mathbf{z}\right)  ;A\mathbf{x}\right\}
-Cb\left(  \mathbf{w}\right)  \left\vert A\mathbf{w}\right\vert ^{k}%
\ln\left\vert A\mathbf{w}\right\vert \delta_{\ast}^{\left[  -k-n\right]
}\left(  \mathbf{x}\right)  \,, \label{FP 8}%
\end{equation}
so that the distribution $\mathcal{P}f\left\{  B_{k}\left(  \mathbf{z}\right)
;A\mathbf{x}\right\}  $ has a non trivial delta part while $\mathcal{P}%
f\left\{  B_{k}\left(  \mathbf{z}\right)  ;\mathbf{x}\right\}  $ does not.

When $f_{0}$ is a smooth function defined in $\mathbb{R}^{n}\setminus\left\{
\mathbf{0}\right\}  $ such that the Hadamard regularization exists at the
origin, and $f\in\mathcal{D}^{\prime}\left(  \mathbb{R}^{n}\right)  $ is a
regularization of $f_{0},$ then we call $f_{0}$ the \emph{ordinary part} of
$f.$ Thus, for instance, $-kx_{i}r^{-k-2}$ is the ordinary part of $\left(
\overline{\nabla}_{i}\right)  \mathcal{P}f\left(  r^{-k}\right)  .$

In a similar fashion, one may consider the \emph{polynomial part} of
distributions. Not all distributions have a well defined polynomial part, but
all the elements of $\mathcal{W}_{\mathrm{pre}}\left(  \mathbb{R}^{n}\right)
$ do. Let us start with the case of a distribution that is homogeneous of
degree $q\geq0$ in $\mathbb{R}^{n}\setminus\left\{  \mathbf{0}\right\}  ,$
that is $F_{q}\left(  \mathbf{u}\right)  =A_{q}\left(  \mathbf{v}\right)
s^{q},$ $\mathbf{u}=s\mathbf{v}$ being polar coordinates and $A_{q}%
\in\mathcal{D}^{\prime}\left(  \mathbb{S}\right)  .$ Then we can write $A_{q}$
in terms of spherical harmonics as $A_{q}\left(  \mathbf{v}\right)
=\sum_{m=0}^{\infty}\mathsf{Y}_{m,q}\left(  \mathbf{v}\right)  \,,$ where
$\mathsf{Y}_{m,q}\in\mathcal{H}_{m}.$ Therefore%
\begin{equation}
F_{q}\left(  \mathbf{u}\right)  =E_{q}\left(  \mathbf{u}\right)
+\widetilde{F}_{q}\left(  \mathbf{u}\right)  \,, \label{delta 6}%
\end{equation}
where $E_{q}=\Pi_{\text{pol}}\left(  F_{q}\right)  $ is the homogeneous
polynomial of degree $q$ given as%
\begin{equation}
E_{q}\left(  \mathbf{u}\right)  =\Pi_{\text{pol}}\left(  F_{q}\right)
=(\sum_{k\leq q/2}\mathsf{Y}_{q-2k,q}\left(  \mathbf{v}\right)  )s^{q},
\label{delta 7}%
\end{equation}
and $\widetilde{F}_{q}=F_{q}-E_{q}$ is the polynomial free part of $F_{q}.$

In the general case when $F$ has the asymptotic expansion of the form%
\begin{equation}
F\left(  s\mathbf{v}\right)  \sim\sum_{q=0}^{Q}\left(  A_{q}\left(
\mathbf{v}\right)  +P_{q}\left(  \mathbf{v}\right)  \ln s\right)  s^{q}%
+\sum_{q=1}^{\infty}A_{-q}\left(  \mathbf{v}\right)  s^{-q}, \label{delta 8}%
\end{equation}
then the \emph{polynomial part} of $F$ is the polynomial%
\begin{equation}
\Pi_{\text{pol}}\left(  F\right)  =\sum_{q=0}^{Q}\Pi_{\text{pol}}\left(
A_{q}\left(  \mathbf{v}\right)  s^{q}\right)  \,. \label{delta 9}%
\end{equation}
The \emph{polynomial free part} of $F$ is $F-\Pi_{\text{pol}}\left(  F\right)
.$

It is possible to define the polynomial part of other distributions, not just
those with an asymptotic expansion of the form (\ref{delta 8}), but this
construction is enough for our purposes, since it gives the polynomial part in
$\mathcal{W}_{\mathrm{pre}}\left(  \mathbb{R}^{n}\right)  .$ It should also be
noticed that the polynomial part we constructed is a \emph{radial} polynomial
part, since we have employed polar coordinates.

The polynomial part allows us to understand why $\mathcal{K}_{q}\left(
\mathcal{D}\left(  \mathbb{S}\right)  \right)  =\mathcal{D}_{q}$ and
$\mathcal{K}_{q}\left(  \mathcal{D}^{\prime}\left(  \mathbb{S}\right)
\right)  =\mathcal{D}_{q}^{\prime}$\ for $q<-n.$ In fact we have the following.

\begin{lemma}
\label{Lemma delta}Let $A\in\mathcal{D}\left(  \mathbb{S}\right)  .$ If
$m\in\mathbb{N},$ then $A\in\mathcal{K}_{-\left(  n+m\right)  }\left(
\mathcal{D}\left(  \mathbb{S}\right)  \right)  $ if and only if the function
$A\left(  \mathbf{v}\right)  s^{m}$ is polynomial free. Similarly, if
$A\in\mathcal{D}^{\prime}\left(  \mathbb{S}\right)  ,$ then $A\in
\mathcal{K}_{-\left(  n+m\right)  }\left(  \mathcal{D}^{\prime}\left(
\mathbb{S}\right)  \right)  $ if and only if the distribution $A\left(
\mathbf{v}\right)  s^{m}$ is polynomial free.
\end{lemma}

\begin{proof}
Follows at once from the Propositions \ref{Lemma 5} and \ref{Lemma 6}.
\end{proof}

\subsection{The space $\mathcal{W}\left(  \mathbb{R}^{n}\right)
$\label{Subsection: The space W}}

We can now consider the topology of the spaces $\mathcal{W}_{\mathrm{pre}%
}\left(  \mathbb{R}^{n}\right)  $ and $\mathcal{S}_{\ast,\text{reg}}\left(
\mathbb{R}^{n}\right)  ,$ as well as define the space $\mathcal{W}\left(
\mathbb{R}^{n}\right)  .$

The space $\mathcal{S}_{\ast,\text{reg}}\left(  \mathbb{R}^{n}\right)  $
admits the representation%
\begin{equation}
\mathcal{S}_{\ast,\text{reg}}\left(  \mathbb{R}^{n}\right)  =\mathcal{S}%
_{\ast,\text{ord}}\left(  \mathbb{R}^{n}\right)  \oplus\mathcal{D}_{\left\{
\mathbf{0}\right\}  }^{\prime}\left(  \mathbb{R}^{n}\right)  \,, \label{W 1}%
\end{equation}
where $\mathcal{D}_{\left\{  \mathbf{0}\right\}  }^{\prime}\left(
\mathbb{R}^{n}\right)  $ is the space of distributions with support at the
origin and where $\mathcal{S}_{\ast,\text{ord}}\left(  \mathbb{R}^{n}\right)
$\ is the space of ordinary parts of regularizations of thick test functions.
Clearly the $\mathcal{P}f$ operator is an isomorphism of $\mathcal{S}_{\ast
}\left(  \mathbb{R}^{n}\right)  $ onto $\mathcal{S}_{\ast,\text{ord}}\left(
\mathbb{R}^{n}\right)  .$ We define the topology of $\mathcal{S}%
_{\ast,\text{ord}}\left(  \mathbb{R}^{n}\right)  $ by asking $\mathcal{P}f$ to
be a topological isomorphism. The space $\mathcal{D}_{\left\{  \mathbf{0}%
\right\}  }^{\prime}\left(  \mathbb{R}^{n}\right)  $ has a topology as a
closed subspace of $\mathcal{S}^{\prime}\left(  \mathbb{R}^{n}\right)  .$ The
topology of $\mathcal{S}_{\ast,\text{reg}}\left(  \mathbb{R}^{n}\right)  $ is
the direct sum topology. Notice that the topology of $\mathcal{S}%
_{\ast,\text{reg}}\left(  \mathbb{R}^{n}\right)  $ is stronger but not equal
to the subspace topology inherited from $\mathcal{S}^{\prime}\left(
\mathbb{R}^{n}\right)  .$ We can now complete the Theorem \ref{Theorem b 1}%
:\ \textsl{The Fourier transform is a topological isomorphism of the spaces}
$\mathcal{S}_{\ast,\text{reg}}\left(  \mathbb{R}^{n}\right)  $ \textsl{and}
$\mathcal{W}_{\mathrm{pre}}\left(  \mathbb{R}^{n}\right)  .$

We now define the space $\mathcal{W}\left(  \mathbb{R}^{n}\right)  .$

\begin{definition}
\label{Definition W 1}The space $\mathcal{W}\left(  \mathbb{R}^{n}\right)  $
is formed by the polynomial free elements of $\mathcal{W}_{\mathrm{pre}%
}\left(  \mathbb{R}^{n}\right)  ,$ with the subspace topology. Explicitly,
$\Phi\in\mathcal{W}$ if it is smooth in $\mathbb{R}^{n}$ and at infinity it
has an asymptotic expansion%
\begin{equation}
\Phi\left(  s\mathbf{v}\right)  \sim\sum_{q=0}^{Q}\left(  A_{q}\left(
\mathbf{v}\right)  +P_{q}\left(  \mathbf{v}\right)  \ln s\right)  s^{q}%
+\sum_{q=1}^{\infty}A_{-q}\left(  \mathbf{v}\right)  s^{-q}, \label{Wdef}%
\end{equation}
where $A_{q}\in\mathcal{D}_{q}$ for $q\in\mathbb{Z}$ and the $P_{q}%
\in\mathcal{P}_{q}$ are homogeneous polynomials of degree $q.$
\end{definition}

The space $\mathcal{W}\left(  \mathbb{R}^{n}\right)  $\ is exactly the space
needed to define the Fourier transform of thick test functions; the condition
$A_{q}\in\mathcal{D}_{q}$ in the expansion (\ref{Wdef}), which\ is equivalent
to the fact that $\Phi$ is polynomial free, will play a very important role in
the behavior of the Fourier transform of thick distributions. Notice in fact
that
\begin{equation}
\mathcal{W}_{\mathrm{pre}}\left(  \mathbb{R}^{n}\right)  =\mathcal{W}\left(
\mathbb{R}^{n}\right)  \oplus\mathcal{P}\left(  \mathbb{R}^{n}\right)
\,,\label{W 2}%
\end{equation}
as topological vector spaces. Therefore the space $\mathcal{W}\left(
\mathbb{R}^{n}\right)  $ can also be constructed as a quotient space. Namely,
if we define the equivalence relation $F\sim G$ if $F-G$ is a polynomial, then
$\mathcal{W}\left(  \mathbb{R}^{n}\right)  $ is canonically isomorphic to
$\mathcal{W}_{\mathrm{pre}}\left(  \mathbb{R}^{n}\right)  /\sim.$ Similarly,
if we consider the equivalence relation $f\sim g$ when $\operatorname*{supp}%
\left(  f-g\right)  \subset\left\{  \mathbf{0}\right\}  $ in $\mathcal{S}%
_{\ast,\text{reg}}\left(  \mathbb{R}^{n}\right)  ,$ then $\mathcal{S}_{\ast
}\left(  \mathbb{R}^{n}\right)  \simeq\mathcal{S}_{\ast,\text{ord}}\left(
\mathbb{R}^{n}\right)  \simeq\mathcal{S}_{\ast,\text{reg}}\left(
\mathbb{R}^{n}\right)  /\sim.\smallskip$

When $\phi\in\mathcal{S}_{\ast}\left(  \mathbb{R}^{n}\right)  $ we shall
denote by $\mathcal{F}_{\ast,\text{t}}\left(  \phi\right)  $ the element
$\Pi_{\mathcal{W}_{\text{pre}},\mathcal{W}}\left(  \mathcal{F}\left(
\mathcal{P}f\left(  \phi\right)  \right)  \right)  $\footnote{In general
$\mathcal{F}\left(  \mathcal{P}f\left(  \phi\right)  \right)  $ does not
belong to $\mathcal{W}.$} of $\mathcal{W}\left(  \mathbb{R}^{n}\right)  ,$ and
call it the thick Fourier transform of $\phi.$ We can also define a Fourier
transform in $\mathcal{W}\left(  \mathbb{R}^{n}\right)  ,$ $\mathcal{F}%
_{\text{t}}^{\ast}:\mathcal{W}\left(  \mathbb{R}^{n}\right)  \longrightarrow
\mathcal{S}_{\ast}\left(  \mathbb{R}^{n}\right)  ,$ as%
\begin{equation}
\mathcal{F}_{\text{t}}^{\ast}\left\{  \Phi\left(  \mathbf{u}\right)
;\mathbf{x}\right\}  =\left(  2\pi\right)  ^{n}\mathcal{F}_{\ast,\text{t}%
}^{-1}\left\{  \Phi\left(  \mathbf{u}\right)  ;-\mathbf{x}\right\}  \,.
\label{W 3}%
\end{equation}
We immediately obtain the following important result.

\begin{theorem}
\label{Theorem W 1}The thick Fourier transform $\mathcal{F}_{\ast,\text{t}}$
is a topological isomorphism of $\mathcal{S}_{\ast}\left(  \mathbb{R}%
^{n}\right)  $ onto $\mathcal{W}\left(  \mathbb{R}^{n}\right)  .$ The thick
Fourier transform $\mathcal{F}_{\text{t}}^{\ast}$ is a topological isomorphism
of $\mathcal{W}\left(  \mathbb{R}^{n}\right)  $ onto $\mathcal{S}_{\ast
}\left(  \mathbb{R}^{n}\right)  .$
\end{theorem}

\section{The space $\mathcal{W}^{\prime}\left(  \mathbb{R}_{\text{c}}%
^{n}\right)  $\label{Section: The space W'}}

In this section we shall consider the distributions of the space
$\mathcal{W}^{\prime}\left(  \mathbb{R}^{n}\right)  .$ The first thing we
would like to point out is that the functions of $\mathcal{W}\left(
\mathbb{R}^{n}\right)  $ are smooth functions in $\mathbb{R}^{n}$ with a
special type of thick behavior at infinity; therefore the elements of
$\mathcal{W}^{\prime}\left(  \mathbb{R}^{n}\right)  $ are actually
distributions over the space $\mathbb{R}_{\text{c}}^{n}=\mathbb{R}^{n}%
\cup\left\{  \mathbf{\infty}\right\}  ,$ the one point compactification of
$\mathbb{R}^{n}.$ From now on we shall also employ the more informative
notation $\mathcal{W}^{\prime}\left(  \mathbb{R}_{\text{c}}^{n}\right)  $ when
we want to call attention to the dimension $n$ and the simpler notation
$\mathcal{W}^{\prime}$ when no explicit mention of $n$ is needed. The elements
of $\mathcal{W}^{\prime}$ shall be called $sl-$\emph{thick distributions,}
since the thick test functions of $\mathcal{W}$ have a special type of
logarithmic expansion at infinity.

Several distributions defined on $\mathbb{R}^{n}$ admit canonical extensions
to $\mathcal{W}^{\prime}\left(  \mathbb{R}_{\text{c}}^{n}\right)  .$ Indeed,
if $\mathcal{W}\left(  \mathbb{R}^{n}\right)  \subset\mathcal{A}\left(
\mathbb{R}^{n}\right)  $ continuously and with dense image, where
$\mathcal{A}\left(  \mathbb{R}^{n}\right)  $ is a space of test functions,
then $\mathcal{A}^{\prime}\left(  \mathbb{R}^{n}\right)  $ is canonically
imbedded into $\mathcal{W}^{\prime}\left(  \mathbb{R}_{\text{c}}^{n}\right)
.$ The simplest case is when $\mathcal{A}\left(  \mathbb{R}^{n}\right)
=\mathcal{E}\left(  \mathbb{R}^{n}\right)  ,$ the space of all smooth
functions in $\mathbb{R}^{n},$ which gives that each distribution of compact
support, $f\in\mathcal{E}^{\prime}\left(  \mathbb{R}^{n}\right)  $ admits a
canonical extension to $\mathcal{W}^{\prime}\left(  \mathbb{R}_{\text{c}}%
^{n}\right)  ,$ namely one whose support in $\mathbb{R}_{\text{c}}^{n}$ is
precisely the original support of $f,$%
\begin{equation}
\left\langle f,\Phi\right\rangle _{\mathcal{W}^{\prime}\times\mathcal{W}%
}=\left\langle f,\Phi\right\rangle _{\mathcal{E}^{\prime}\times\mathcal{E}}\,.
\label{Wp 0.1}%
\end{equation}
Actually we can also take $\mathcal{A}\left(  \mathbb{R}^{n}\right)
=\mathcal{K}\left(  \mathbb{R}^{n}\right)  ,$ so that any distribution
$f\in\mathcal{K}^{\prime}\left(  \mathbb{R}^{n}\right)  $ admits a canonical
extension to $\mathcal{W}^{\prime}\left(  \mathbb{R}_{\text{c}}^{n}\right)  ,$
given by the Ces\`{a}ro evaluation,%
\begin{equation}
\left\langle f,\Phi\right\rangle _{\mathcal{W}^{\prime}\times\mathcal{W}%
}=\left\langle f,\Phi\right\rangle \ \ \ \left(  \text{C}\right)  \label{Wp 1}%
\end{equation}
since $\left\langle f,\Phi\right\rangle $ $\left(  \text{C}\right)  $ exists
whenever $\Phi\in\mathcal{K}\left(  \mathbb{R}^{n}\right)  $ \cite{GreenBook}%
\ and $\mathcal{W}\left(  \mathbb{R}^{n}\right)  \subset\mathcal{K}\left(
\mathbb{R}^{n}\right)  .$ We shall employ the same notation for both the
distribution of $\mathcal{K}^{\prime}\left(  \mathbb{R}^{n}\right)  $ and its
canonical extension to $\mathcal{W}^{\prime}\left(  \mathbb{R}_{\text{c}}%
^{n}\right)  .$ On the other hand, $\mathcal{W}\left(  \mathbb{R}^{n}\right)
$ is not contained in $\mathcal{S}\left(  \mathbb{R}^{n}\right)  ,$ and this
means that tempered distributions do not have \emph{canonical} extensions in
$\mathcal{W}^{\prime}\left(  \mathbb{R}_{\text{c}}^{n}\right)  .$ In fact, it
is not hard to see that actually all elements of $\mathcal{S}^{\prime}\left(
\mathbb{R}^{n}\right)  $ have\emph{ many} extensions to $\mathcal{W}^{\prime
}\left(  \mathbb{R}_{\text{c}}^{n}\right)  ,$ but it is not possible to
construct a continuous extension procedure, similarly to the situation
explained in \cite{E2003}.

Another important class of $sl-$thick distributions are the thick deltas at infinity.

\begin{definition}
\label{Def Wp 1}If $G\in\mathcal{D}_{q}^{\prime}$ then we define $G\left(
\mathbf{v}\right)  \delta_{\infty}^{\left[  q\right]  },$ the thick delta at
infinity of order $q$ as
\begin{equation}
\left\langle G\left(  \mathbf{v}\right)  \delta_{\infty}^{\left[  q\right]
},\Phi\right\rangle _{\mathcal{W}^{\prime}\times\mathcal{W}}=\frac{1}%
{C}\left\langle G,A_{q}\right\rangle _{\mathcal{D}_{q}^{\prime}\times
\mathcal{D}_{q}}\,, \label{Wp 3}%
\end{equation}
if $\Phi\in\mathcal{W}$ has the asymptotic expansion (\ref{Wdef}). Similarly,
if $H\in\mathcal{P}_{q}^{\prime}=\mathcal{P}_{q}$ then we define $H\left(
\mathbf{v}\right)  \delta_{\ln,\infty}^{\left[  q\right]  }$\ the thick
logarithmic delta of order $q$ at infinity as%
\begin{equation}
\left\langle H\left(  \mathbf{v}\right)  \delta_{\ln,\infty}^{\left[
q\right]  },\Phi\right\rangle _{\mathcal{W}^{\prime}\times\mathcal{W}}%
=\frac{1}{C}\left\langle H,P_{q}\right\rangle _{\mathcal{P}_{q}^{\prime}%
\times\mathcal{P}_{q}}\,. \label{Wp 4}%
\end{equation}

\end{definition}

Sometimes one may construct extensions of a tempered distribution $g$ by
considering a finite part at infinity\footnote{Clearly the finite part at
infinity does \emph{not exist} for all $g\in\mathcal{S}^{\prime}\left(
\mathbb{R}^{n}\right)  .$}, a construction we shall now denote as
$\mathcal{P}f_{\mathcal{W}}\left(  g\right)  ,$ or later simply as
$\mathcal{P}f\left(  g\right)  $ if there is no danger of confusion. Consider
for example the distribution $\mathcal{P}f\left(  s^{\lambda}\right)  ,$
$s=\left\vert \mathbf{u}\right\vert ,$ of $\mathcal{S}^{\prime}\left(
\mathbb{R}^{n}\right)  :$ this tempered distribution yields the $sl-$thick
distribution $\mathcal{P}f_{\mathcal{W}}\left(  s^{\lambda}\right)  $ obtained
from the generally divergent integral $\int_{\mathbb{R}^{n}}s^{\lambda}%
\Phi\left(  \mathbf{u}\right)  \,d\mathbf{u},$ $\Phi\in\mathcal{W},$ by taking
the radial finite part at $\mathbf{0},$ or at $\mathbf{\infty},$ or at both.
Using the ideas of the Example \ref{example 1} we can see the structure of
$\mathcal{P}f_{\mathcal{W}}\left(  s^{\lambda}\right)  .$

\begin{lemma}
\label{Lemma Wp}The parametric $sl-$thick distribution $\mathcal{P}%
f_{\mathcal{W}}\left(  s^{\lambda}\right)  $ is a meromorphic function of
$\lambda,$ analytic in the region $\left(  \mathbb{C}\setminus\mathbb{Z}%
\right)  \cup\left\{  0,2,4,\ldots\right\}  ,$ with simple poles at
$\lambda=m,$ $m\in\left\{  -n-1,-n-3,-n-5,\ldots\right\}  \cup\left\{
-1,-2,\ldots,1-n\right\}  \cup\left\{  1,3,5,\ldots\right\}  ,$ the residues
at these poles being%
\begin{equation}
\operatorname*{Res}\nolimits_{\lambda=m}\mathcal{P}f_{\mathcal{W}}\left(
s^{\lambda}\right)  =-C\delta_{\infty}^{\left[  -n-m\right]  }\left(
\mathbf{u}\right)  \,, \label{Wp 5p}%
\end{equation}
and double poles at $\lambda=m,$ $m=-n-2q\in\left\{  -n,-n-2,-n-4,\ldots
\right\}  $ with singular part%
\begin{equation}
\frac{C\delta_{\ln,\infty}^{\left[  2q\right]  }\left(  \mathbf{u}\right)
}{\left(  \lambda-m\right)  ^{2}}+\frac{c_{q,n}\nabla^{2q}\delta\left(
\mathbf{u}\right)  }{\left(  2q\right)  !\left(  \lambda-m\right)  }\,.
\label{Wp 5q}%
\end{equation}
The finite part of $\mathcal{P}f_{\mathcal{W}}\left(  s^{\lambda}\right)  $ at
any pole $\lambda=m$ is precisely $\mathcal{P}f_{\mathcal{W}}\left(
s^{m}\right)  .$
\end{lemma}

Many of the constructions that we have discussed can also be done in the space
$\mathcal{W}_{\text{pre}}^{\prime}.$ Notice, however, that several
distributions of $\mathcal{W}_{\text{pre}}^{\prime}$ could vanish in
$\mathcal{W}$ so that their projection to $\mathcal{W}^{\prime}$ could be
zero. For instance, the plain thick delta $\delta_{\infty}^{\left[  0\right]
}$ is not zero in $\mathcal{W}_{\text{pre}}^{\prime}$ but it is zero in
$\mathcal{W}^{\prime}.$ If one considers the finite part $\mathcal{P}%
f_{\mathcal{W}_{\text{pre}}}\left(  s^{\lambda}\right)  $ then it would not be
analytic at $\lambda=0,2,4,\ldots;$ for instance, it has a simple pole at
$\lambda=0$ with residue $-C\delta_{\infty}^{\left[  0\right]  }.$

One of the consequences of the fact that $\mathcal{W}^{\prime}$ is a space
over the compact space $\mathbb{R}_{\text{c}}^{n}$ is that several of the
usual operations on $sl-$thick distributions could have additional terms at
infinity. This is the case for the linear changes of variables and for the
multiplications by polynomials. Curiously, however, derivatives in
$\mathcal{W}^{\prime}$ can be defined in the standard way by duality, since
the derivative operators send $\mathcal{W}$ to $\mathcal{W},$%
\begin{equation}
\left\langle \nabla_{j}\left(  F\right)  ,\Phi\right\rangle =-\left\langle
F,\nabla_{j}\left(  \Phi\right)  \right\rangle \,,\ \ \ F\in\mathcal{W}%
^{\prime},\Phi\in\mathcal{W}\,. \label{Wp 5}%
\end{equation}

\subsection{Linear changes of variables in $\mathcal{W}^{\prime}%
$\label{Subsection: Linear changes of variables in}}

Let $A$ be a non-singular $n\times n$ matrix. If $\Phi\in\mathcal{W}$ then the
function $\Phi_{A}$ given by $\Phi_{A}\left(  \mathbf{u}\right)  =\Phi\left(
A\mathbf{u}\right)  $ does not belong to $\mathcal{W},$ in general, but it
belongs to $\mathcal{W}_{\text{pre}}.$ Therefore we define the function of
$\mathcal{W}$ obtained by the change of variables, $\tau_{A}^{\mathcal{W}%
}\left(  \Phi\right)  $ as%
\begin{equation}
\tau_{A}^{\mathcal{W}}\left(  \Phi\right)  =\Pi_{\mathcal{W}_{\text{pre}%
},\mathcal{W}}\left(  \Phi_{A}\right)  \,. \label{lc 1}%
\end{equation}
We can then define the change of variables in $sl-$thick distributions by duality.

\begin{definition}
\label{Definition cl 1}Let $A$ be a non-singular $n\times n$ matrix. If
$F\in\mathcal{W}^{\prime}$ then the distribution $\tau_{A}^{\mathcal{W}%
^{\prime}}\left(  F\right)  ,$ the $sl-$thick version of $F\left(
A\mathbf{u}\right)  ,$ is defined as%
\begin{equation}
\left\langle \tau_{A}^{\mathcal{W}^{\prime}}\left(  F\right)  ,\Phi
\right\rangle =\frac{1}{\left\vert \det\left(  A\right)  \right\vert
}\left\langle F,\tau_{A^{-1}}^{\mathcal{W}}\left(  \Phi\right)  \right\rangle
\,. \label{lc 2}%
\end{equation}

\end{definition}

It is important to observe that if $F\in\mathcal{K}^{\prime}\left(
\mathbb{R}^{n}\right)  $ then it has a canonical extension to $\mathcal{W}%
^{\prime}\left(  \mathbb{R}_{\text{c}}^{n}\right)  ,$ and the restriction of
$\tau_{A}^{\mathcal{W}^{\prime}}\left(  F\right)  \left(  \mathbf{x}\right)  $
to $\mathbb{R}^{n}$ is precisely $F\left(  A\mathbf{x}\right)  $ but in
general $\tau_{A}^{\mathcal{W}^{\prime}}\left(  F\right)  \left(
\mathbf{x}\right)  $ is \emph{not} the canonical extension of $F\left(
A\mathbf{x}\right)  .$ A simple example is provided by the delta function at
the origin, $\delta\left(  \mathbf{x}\right)  ,$ and the change $A=tI$ for
$t\neq0.$ We have $\delta\left(  t\mathbf{x}\right)  =\left\vert t\right\vert
^{-n}\delta\left(  \mathbf{x}\right)  ,$ of course, but%
\begin{equation}
\tau_{tI}^{\mathcal{W}^{\prime}}\left(  \delta\right)  \left(  \mathbf{x}%
\right)  =\left\vert t\right\vert ^{-n}\delta\left(  \mathbf{x}\right)
-\left\vert t\right\vert ^{-n}\ln t\,\delta_{\infty,\ln}^{\left[  0\right]
}\left(  \mathbf{x}\right)  \,. \label{lc 3}%
\end{equation}
Interestingly, if $A$ is an orthogonal matrix, in particular if it is a
rotation, and $F\in\mathcal{K}^{\prime}\left(  \mathbb{R}^{n}\right)  $\ then
the canonical extension of $F\left(  A\mathbf{x}\right)  $ is precisely
$\tau_{A}^{\mathcal{W}^{\prime}}\left(  F\right)  \left(  \mathbf{x}\right)
.$ Therefore we give the following definitions.

\begin{definition}
\label{Definition lc 2}A $sl-$thick distribution $F\in\mathcal{W}^{\prime}$ is
called radial if $\tau_{A}^{\mathcal{W}^{\prime}}\left(  F\right)  =F$ for all
orthogonal matrices $A.$ We say that $F$ is homogeneous of order $\lambda$ if
\begin{equation}
\tau_{tI}^{\mathcal{W}^{\prime}}\left(  F\right)  \left(  \mathbf{x}\right)
=t^{\lambda}F\left(  \mathbf{x}\right)  \,,\ \ \ \ t>0\,. \label{lc 4}%
\end{equation}

\end{definition}

Notice that a distribution $F\in\mathcal{K}^{\prime}\left(  \mathbb{R}%
^{n}\right)  $\ is radial if and only if its canonical extension is, but
(\ref{lc 3}) shows that a corresponding result does not hold for homogeneous
distributions. On the other hand, a distribution of the form $G\left(
\mathbf{v}\right)  \delta_{\infty}^{\left[  q\right]  }$ is radial if and only
if $G$ is constant, where we observe that the plain thick delta at infinity
$\delta_{\infty}^{\left[  q\right]  }$ is a non zero $sl-$thick distribution
for $q\neq0,2,4,\ldots$ and $q\neq-n,-n-2,-n-4,\ldots.$ Furthermore, since the
plain thick logarithmic deltas at infinity $\delta_{\ln,\infty}^{\left[
1\right]  },\delta_{\ln,\infty}^{\left[  3\right]  },\delta_{\ln,\infty
}^{\left[  5\right]  },\ldots$ vanish, the distributions $c\delta_{\ln,\infty
}^{\left[  q\right]  }$ for $q=0,2,4,\ldots$ and $c$ constant are the radial
distributions of the form $G\left(  \mathbf{v}\right)  \delta_{\ln,\infty
}^{\left[  q\right]  }.$

It is useful to know the $sl-$thick radial homogeneous distributions.

\begin{proposition}
\label{Prop. lc 1}Let $\lambda\in\mathbb{C}.$ Then the set of $sl-$thick
radial homogeneous distributions of order $\lambda$ form a vector space of
dimension $1,$ generated by the distribution%
\begin{equation}
\mathcal{P}f_{\mathcal{W}}\left(  s^{\lambda}\right)  \text{ for }\lambda
\in\left(  \mathbb{C}\setminus\mathbb{Z}\right)  \cup\left\{  0,2,4,\ldots
\right\}  \,, \label{lc 7}%
\end{equation}%
\begin{align}
&  \ \ \ \ \ \ \ \ \ \ \ \ \ \ \delta_{\infty}^{\left[  -n-m\right]
}\ \ \ \text{for }\lambda=m,\label{lc 5}\\
&  m\in\left\{  -n-1,-n-3,-n-5,\ldots\right\}  \cup\left\{  1,2,\ldots
,n-1\right\}  \cup\left\{  1,3,5,\ldots\right\}  \,,\nonumber
\end{align}%
\begin{equation}
\delta_{\ln,\infty}^{\left[  -n-m\right]  }\text{ \ \ for }\lambda
=m,\ \ m\in\left\{  -n,-n-2,-n-4,\ldots\right\}  \,. \label{lc 6}%
\end{equation}

\end{proposition}

\subsubsection{Multiplication by polynomials in $\mathcal{W}^{\prime}%
$\label{Subsection: Multiplication by polynomials in}}

In general if $\Phi\in\mathcal{W}$ then $u_{j}\Phi\left(  \mathbf{u}\right)  $
is in $\mathcal{W}_{\text{pre}}$\ but it does not belong to $\mathcal{W}.$
Therefore we define the multiplication operator
\begin{equation}
M_{u_{j}}^{\mathcal{W}}:\mathcal{W}\longrightarrow\mathcal{W}\,,\ \ \ M_{u_{j}%
}\left(  \Phi\right)  =\Pi_{\mathcal{W}_{\text{pre}},\mathcal{W}}\left(
u_{j}\Phi\right)  \,, \label{mu 1}%
\end{equation}
and by duality the operator $M_{u_{j}}^{\mathcal{W}^{\prime}}:\mathcal{W}%
^{\prime}\longrightarrow\mathcal{W}^{\prime}$ as
\begin{equation}
\left\langle M_{u_{j}}^{\mathcal{W}^{\prime}}\left(  F\right)  ,\Phi
\right\rangle =\left\langle F,M_{u_{j}}^{\mathcal{W}}\left(  \Phi\right)
\right\rangle \,. \label{mu 2}%
\end{equation}
The multiplication operators $M_{p}^{\mathcal{W}}$ and $M_{p}^{\mathcal{W}%
^{\prime}},$ where $p$ is a polynomial, can be defined in a similar way.

\begin{example}
\label{Example mult}Sometimes $M_{u_{j}}^{\mathcal{W}^{\prime}}\left(
F\right)  $ resembles a standard multiplication, as in the product formula%
\begin{equation}
M_{u_{j}}^{\mathcal{W}^{\prime}}\left(  \mathcal{P}f_{\mathcal{W}}\left(
s^{\lambda}\right)  \right)  =\mathcal{P}f_{\mathcal{W}}\left(  u_{j}%
s^{\lambda}\right)  \,, \label{mu 3}%
\end{equation}
but sometimes extra terms at infinity appear, as in the formula%
\begin{equation}
M_{u_{j}}^{\mathcal{W}^{\prime}}\left(  \delta\left(  \mathbf{u}\right)
\right)  =-\omega_{j}\delta_{\infty}^{\left[  -1\right]  }\left(
\mathbf{u}\right)  \,. \label{mu 4}%
\end{equation}

\end{example}

\section{The Fourier transform of thick
distributions\label{Section: The Fourier transform of thick distributions}}

The Fourier transform of thick tempered distributions $f\in\mathcal{S}_{\ast
}^{\prime}\left(  \mathbb{R}^{n}\right)  ,$ $\mathcal{F}_{\ast}\left(
f\right)  \in\mathcal{W}^{\prime}\left(  \mathbb{R}_{\text{c}}^{n}\right)  $
can now be defined in the usual way,%
\begin{equation}
\left\langle \mathcal{F}_{\ast}\left\{  f\left(  \mathbf{x}\right)
;\mathbf{u}\right\}  ,\Phi\left(  \mathbf{u}\right)  \right\rangle
=\left\langle f\left(  \mathbf{x}\right)  ,\mathcal{F}_{\text{t}}^{\ast
}\left\{  \Phi\left(  \mathbf{u}\right)  ;\mathbf{x}\right\}  \right\rangle
\,,\ \ \ \Phi\in\mathcal{W}\left(  \mathbb{R}_{\text{c}}^{n}\right)  \,.
\label{W 4}%
\end{equation}
Similarly, the Fourier transform of distributions $G\in\mathcal{W}^{\prime
}\left(  \mathbb{R}_{\text{c}}^{n}\right)  ,$ $\mathcal{F}^{\ast}\left(
G\right)  \in\mathcal{S}_{\ast}^{\prime}\left(  \mathbb{R}^{n}\right)  $ is
defined as%
\begin{equation}
\left\langle \mathcal{F}^{\ast}\left\{  G\left(  \mathbf{u}\right)
;\mathbf{x}\right\}  ,\phi\left(  \mathbf{x}\right)  \right\rangle
=\left\langle G\left(  \mathbf{u}\right)  ,\mathcal{F}_{\ast,\text{t}}\left\{
\phi\left(  \mathbf{x}\right)  ;\mathbf{u}\right\}  \right\rangle
\,,\ \ \ \phi\in\mathcal{S}_{\ast}\left(  \mathbb{R}^{n}\right)  \,.
\label{W 5}%
\end{equation}

\begin{theorem}
\label{Theorem Wprime}The thick Fourier transform $\mathcal{F}_{\ast}$ is a
topological isomorphism of $\mathcal{S}_{\ast}^{\prime}\left(  \mathbb{R}%
^{n}\right)  $ onto $\mathcal{W}^{\prime}\left(  \mathbb{R}_{\text{c}}%
^{n}\right)  .$ The thick Fourier transform $\mathcal{F}^{\ast}$ is a
topological isomorphism of $\mathcal{W}^{\prime}\left(  \mathbb{R}_{\text{c}%
}^{n}\right)  $ onto $\mathcal{S}_{\ast}^{\prime}\left(  \mathbb{R}%
^{n}\right)  .$
\end{theorem}

The properties of the Fourier transform of thick distributions are similar to
those of the transform in $\mathcal{S}^{\prime}\left(  \mathbb{R}^{n}\right)
$ but one must remember that the operations in $\mathcal{W}^{\prime}\left(
\mathbb{R}_{\text{c}}^{n}\right)  $ may or may not be the standard ones. We
have,
\begin{equation}
\mathcal{F}_{\ast}\left\{  f\left(  A\mathbf{x}\right)  ;\mathbf{u}\right\}
=\frac{1}{\left\vert \det A\right\vert }\tau_{A^{-1}}^{\mathcal{W}^{\prime}%
}\left(  \mathcal{F}_{\ast}\left\{  f\left(  \mathbf{x}\right)  ;\mathbf{u}%
\right\}  \right)  \,, \label{FW 1}%
\end{equation}
for $A$ a non-singular matrix, and in particular, if $t\neq0$%
\begin{equation}
\mathcal{F}_{\ast}\left\{  f\left(  t\mathbf{x}\right)  ;\mathbf{u}\right\}
=t^{-n}\tau_{t^{-n}I}^{\mathcal{W}^{\prime}}\left(  \mathcal{F}_{\ast}\left\{
f\left(  \mathbf{x}\right)  ;\mathbf{u}\right\}  \right)  \,. \label{Fw 2}%
\end{equation}
It follows that $\mathcal{F}_{\ast}$ and $\mathcal{F}^{\ast}$ send radial
thick distributions to radial thick distributions, and homogeneous
distributions of degree $\lambda$ to homogeneous distributions of degree
$-n-\lambda.$ We also have the usual interchange of multiplication and
differentiation,%
\begin{equation}
\mathcal{F}_{\ast}\left\{  x_{j}f\left(  \mathbf{x}\right)  ;\mathbf{u}%
\right\}  =-i\nabla_{u_{j}}\mathcal{F}_{\ast}\left\{  f\left(  \mathbf{x}%
\right)  ;\mathbf{u}\right\}  \,, \label{FW 3}%
\end{equation}%
\begin{equation}
\mathcal{F}_{\ast}\left\{  \nabla_{x_{j}}f\left(  \mathbf{x}\right)
;\mathbf{u}\right\}  =-iM_{u_{j}}^{\mathcal{W}^{\prime}}\mathcal{F}_{\ast
}\left\{  f\left(  \mathbf{x}\right)  ;\mathbf{u}\right\}  \,, \label{FW 4}%
\end{equation}
where the modified multiplication operator $M_{u_{j}}^{\mathcal{W}^{\prime}}%
$\ is given by (\ref{mu 2}). The formulas for the inverse transforms are a
variant of the usual ones,%
\begin{equation}
\left(  \mathcal{F}^{\ast}\right)  ^{-1}\left\{  f\left(  \mathbf{x}\right)
;\mathbf{u}\right\}  =\frac{1}{\left(  2\pi\right)  ^{n}}\mathcal{F}_{\ast
}\left\{  f\left(  \mathbf{x}\right)  ;-\mathbf{u}\right\}  \,, \label{FW 5}%
\end{equation}%
\begin{equation}
\left(  \mathcal{F}_{\ast}\right)  ^{-1}\left\{  F\left(  \mathbf{u}\right)
;\mathbf{x}\right\}  =\frac{1}{\left(  2\pi\right)  ^{n}}\mathcal{F}^{\ast
}\left\{  F\left(  \mathbf{u}\right)  ;-\mathbf{x}\right\}  \,. \label{FW 6}%
\end{equation}
Another important property is that the Fourier transforms $\mathcal{F}_{\ast}$
or $\mathcal{F}^{\ast}$ of extensions of distributions of $\mathcal{S}%
^{\prime}\left(  \mathbb{R}^{n}\right)  $ to $\mathcal{S}_{\ast}^{\prime
}\left(  \mathbb{R}^{n}\right)  $ or $\mathcal{W}^{\prime}\left(
\mathbb{R}_{\text{c}}^{n}\right)  $ are extensions of the Fourier transform,
that is%
\begin{equation}
\Pi_{\mathcal{W}^{\prime},\mathcal{S}^{\prime}}\mathcal{F}_{\ast}\left\{
f\left(  \mathbf{x}\right)  ;\mathbf{u}\right\}  =\mathcal{F}\left\{
\Pi_{\mathcal{S}_{\ast}^{\prime},\mathcal{S}^{\prime}}f\left(  \mathbf{x}%
\right)  ;\mathbf{u}\right\}  \,, \label{FW 7}%
\end{equation}%
\begin{equation}
\Pi_{\mathcal{S}_{\ast}^{\prime},\mathcal{S}^{\prime}}\mathcal{F}^{\ast
}\left\{  F\left(  \mathbf{u}\right)  ;\mathbf{x}\right\}  =\mathcal{F}%
\left\{  \Pi_{\mathcal{W}^{\prime},\mathcal{S}^{\prime}}F\left(
\mathbf{u}\right)  ;\mathbf{x}\right\}  \,. \label{FW 8}%
\end{equation}

We are now ready to give the Fourier transform of several thick distributions.

\begin{example}
\label{example 2}Let us compute the Fourier transform $\mathcal{F}_{\ast
}\left\{  \delta_{\ast}^{\left[  0\right]  }\left(  \mathbf{x}\right)
;\mathbf{u}\right\}  $ of the plain thick delta function. Since $\delta_{\ast
}^{\left[  0\right]  }\left(  \mathbf{x}\right)  $ is radial and homogenous of
degree $-n,$ its transform is radial and homogeneous of degree $0.$ Also, the
projection of $\delta_{\ast}^{\left[  0\right]  }\left(  \mathbf{x}\right)  $
onto $\mathcal{S}^{\prime}$ is the standard delta function $\delta\left(
\mathbf{x}\right)  ,$ whose transform is the constant function $1.$ From the
Proposition \ref{Prop. lc 1} it follows that the only radial, homogeneous of
degree $0$ $sl-$thick distribution whose projection to $\mathcal{S}^{\prime}$
is the constant distribution $1$ is precisely $\mathcal{P}f_{\mathcal{W}%
}\left(  1\right)  .$ Hence%
\begin{equation}
\mathcal{F}_{\ast}\left\{  \delta_{\ast}^{\left[  0\right]  }\left(
\mathbf{x}\right)  ;\mathbf{u}\right\}  =\mathcal{P}f_{\mathcal{W}}\left(
1\right)  \,. \label{FW 9}%
\end{equation}
A similar argument yields%
\begin{equation}
\mathcal{F}_{\ast}\left\{  \delta_{\ast}^{\left[  2m\right]  }\left(
\mathbf{x}\right)  ;\mathbf{u}\right\}  =\frac{\left(  -1\right)  ^{m}%
\Gamma\left(  m+1/2\right)  \Gamma\left(  n/2\right)  }{\Gamma\left(
m+n/2\right)  \Gamma\left(  1/2\right)  \left(  2m\right)  !}\mathcal{P}%
f_{\mathcal{W}}\left(  s^{2m}\right)  \,, \label{FW 10}%
\end{equation}
and by inversion,%
\begin{equation}
\mathcal{F}^{\ast}\left\{  \mathcal{P}f_{\mathcal{W}}\left(  s^{2m}\right)
;\mathbf{x}\right\}  =\frac{\left(  -1\right)  ^{m}\Gamma\left(  m+n/2\right)
\Gamma\left(  1/2\right)  \left(  2m\right)  !}{\left(  2\pi\right)
^{n}\Gamma\left(  m+1/2\right)  \Gamma\left(  n/2\right)  }\delta_{\ast
}^{\left[  2m\right]  }\left(  \mathbf{x}\right)  \,. \label{FW 11}%
\end{equation}

\end{example}

\begin{example}
\label{Example 3}The ensuing formulas, reminiscent of (\ref{2}), also follow
along the same lines,%
\begin{equation}
\mathcal{F}_{\ast}\left\{  \mathcal{P}f\left(  r^{\lambda}\right)
;\mathbf{u}\right\}  =\frac{\pi^{n/2}2^{\lambda+n}\Gamma\left(  \frac
{\lambda+n}{2}\right)  }{\Gamma\left(  -\frac{\lambda}{2}\right)  }%
\mathcal{P}f_{\mathcal{W}}\left(  s^{-\lambda-n}\right)  \,, \label{FF 1}%
\end{equation}%
\begin{equation}
\mathcal{F}^{\ast}\left\{  \mathcal{P}f_{\mathcal{W}}\left(  s^{\lambda
}\right)  ;\mathbf{x}\right\}  =\frac{\pi^{n/2}2^{\lambda+n}\Gamma\left(
\frac{\lambda+n}{2}\right)  }{\Gamma\left(  -\frac{\lambda}{2}\right)
}\mathcal{P}f\left(  r^{-\lambda-n}\right)  \,, \label{FF 2}%
\end{equation}
whenever $\lambda\in\mathbb{C}\setminus\mathbb{Z}.$ Interestingly,
$\mathcal{P}f_{\mathcal{W}}\left(  s^{\lambda}\right)  $ is analytic at
$0,2,4,\ldots$ so that (\ref{FW 11}) can be recovered by taking the limit as
$\lambda\rightarrow2m$ in the right side of (\ref{FF 2}).
\end{example}

\begin{example}
\label{Example 4}Formulas (\ref{FF 1}) and (\ref{FF 2}) are equalities of
meromorphic functions and thus by considering the residues, finite parts, or
singular parts at the poles of both sides we obtain the Fourier transform of
several thick distributions. Let start with $m\in\left\{
-n-1,-n-3,-n-5,\ldots\right\}  \cup\left\{  -1,-2,\ldots,1-n\right\}
\cup\left\{  1,3,5,\ldots\right\}  ,$ so that $\lambda=m$ is a simple pole of
the function in (\ref{FF 2}). From the Lemma \ref{Lemma Wp} the residue of the
left side is $\mathcal{F}^{\ast}\left\{  -C\delta_{\infty}^{\left[
-n-m\right]  }\left(  \mathbf{u}\right)  ;\mathbf{x}\right\}  ,$ while if we
recall \cite[(4.13)]{YE2}\ that $\operatorname*{Res}_{\mu=k}\mathcal{P}%
f\left(  r^{\mu}\right)  =C\delta_{\ast}^{\left[  -k-n\right]  }\left(
\mathbf{x}\right)  ,$ we obtain the residue of the right side as $Cg\left(
m\right)  \delta_{\ast}^{\left[  m\right]  }\left(  \mathbf{x}\right)  $ where%
\begin{equation}
g\left(  \lambda\right)  =\frac{\pi^{n/2}2^{\lambda+n}\Gamma\left(
\frac{\lambda+n}{2}\right)  }{\Gamma\left(  -\frac{\lambda}{2}\right)  }\,.
\label{FF 3}%
\end{equation}
Therefore%
\begin{equation}
\mathcal{F}^{\ast}\left\{  \delta_{\infty}^{\left[  -n-m\right]  }\left(
\mathbf{u}\right)  ;\mathbf{x}\right\}  =-g\left(  m\right)  \delta_{\ast
}^{\left[  m\right]  }\left(  \mathbf{x}\right)  \,, \label{FF 4}%
\end{equation}
and by inversion,%
\begin{equation}
\mathcal{F}_{\ast}\left\{  \delta_{\ast}^{\left[  m\right]  }\left(
\mathbf{x}\right)  ;\mathbf{u}\right\}  =-g\left(  -n-m\right)  \delta
_{\infty}^{\left[  -n-m\right]  }\left(  \mathbf{u}\right)  \,, \label{FF 5}%
\end{equation}
since $g\left(  m\right)  g\left(  -n-m\right)  =\left(  2\pi\right)  ^{n}.$
Similarly, consideration of the finite parts of both sides of (\ref{FF 2})
yields%
\begin{equation}
\mathcal{F}^{\ast}\left\{  \mathcal{P}f_{\mathcal{W}}\left(  s^{m}\right)
;\mathbf{x}\right\}  =g\left(  m\right)  \left\{  \mathcal{P}f\left(
r^{-m-n}\right)  +\chi_{m}\delta_{\ast}^{\left[  m\right]  }\left(
\mathbf{x}\right)  \right\}  \,, \label{FF 6}%
\end{equation}
and%
\begin{equation}
\mathcal{F}_{\ast}\left\{  \mathcal{P}f\left(  r^{-m-n}\right)  ;\mathbf{u}%
\right\}  =g\left(  -n-m\right)  \left\{  \mathcal{P}f_{\mathcal{W}}\left(
s^{m}\right)  +\chi_{-m-n}\delta_{\infty}^{\left[  -n-m\right]  }\left(
\mathbf{u}\right)  \right\}  \,, \label{FF 7}%
\end{equation}
where%
\begin{equation}
\chi_{m}=\chi_{-m-n}=\frac{C}{2}\left(  2\ln2+\psi\left(  \frac{m+n}%
{2}\right)  +\psi\left(  \frac{-m}{2}\right)  \right)  \,. \label{FF 8}%
\end{equation}
Studying the coefficients of order $-2$ at the poles of order $2,$ $m=-n-2q$
for $q\in\mathbb{N}$ gives%
\begin{equation}
\mathcal{F}^{\ast}\left\{  \delta_{\ln,\infty}^{\left[  2q\right]  }\left(
\mathbf{u}\right)  ;\mathbf{x}\right\}  =\frac{\left(  -1\right)  ^{n}%
2^{1-2q}\pi^{n/2}}{q!\Gamma\left(  \frac{n+2q}{2}\right)  }\delta_{\ast
}^{\left[  -n-2q\right]  }\left(  \mathbf{x}\right)  \,, \label{FF 9}%
\end{equation}
and%
\begin{equation}
\mathcal{F}_{\ast}\left\{  \delta_{\ast}^{\left[  -n-2q\right]  }\left(
\mathbf{x}\right)  ;\mathbf{u}\right\}  =\left(  -1\right)  ^{n}2^{n+2q-1}%
\pi^{n/2}q!\Gamma\left(  \frac{n+2q}{2}\right)  \delta_{\ln,\infty}^{\left[
2q\right]  }\left(  \mathbf{u}\right)  \,. \label{FF 9p}%
\end{equation}

\end{example}

We have considered the transform of plain thick deltas so far, now we compute
the Fourier transform of general thick deltas.

\begin{example}
\label{Example 5}Let $\phi\in\mathcal{S}_{\ast},$ with expansion $\sum
_{j=m}^{\infty}a_{j}r^{j}$ at zero and let $\Phi=\mathcal{F}_{\ast,\text{t}%
}\left(  \phi\right)  \in\mathcal{W},$ with expansion $\sum_{q=0}^{n-m}\left(
A_{q}\left(  \mathbf{v}\right)  +P_{q}\left(  \mathbf{v}\right)  \ln s\right)
s^{q}+\sum_{q=1}^{\infty}A_{-q}\left(  \mathbf{v}\right)  s^{-q}$\ at
infinity. Then $A_{q}=\mathfrak{K}_{q}\left(  a_{-n-q}\right)  ,$ therefore if
$G\in\mathcal{D}_{q}^{^{\prime}}$ then
\[
\left\langle G\delta_{\infty}^{\left[  q\right]  },\Phi\right\rangle =\frac
{1}{C}\left\langle G,A_{q}\right\rangle =\frac{1}{C}\left\langle
G,\mathfrak{K}_{q}\left(  a_{-n-q}\right)  \right\rangle =\frac{1}%
{C}\left\langle \mathfrak{K}_{q}\left(  G\right)  ,a_{-n-q}\right\rangle
=\left\langle \mathfrak{K}_{q}\left(  G\right)  \delta_{\ast}^{\left[
-n-q\right]  },\phi\right\rangle \,,
\]
or%
\begin{equation}
\mathcal{F}^{\ast}\left\{  G\left(  \mathbf{v}\right)  \delta_{\infty
}^{\left[  q\right]  }\left(  \mathbf{u}\right)  ;\mathbf{x}\right\}
=\mathfrak{K}_{q}\left\{  G\left(  \mathbf{v}\right)  ;\mathbf{w}\right\}
\delta_{\ast}^{\left[  -n-q\right]  }\left(  \mathbf{x}\right)  \,,
\label{FFp 1}%
\end{equation}
giving the transform of \emph{all} thick deltas at infinity $G\delta_{\infty
}^{\left[  q\right]  },$ for arbitrary $q\in\mathbb{Z},$ since $G$ needs to be
$\mathcal{D}_{q}^{^{\prime}}.$ Similarly, for $q\in\mathbb{N}$%
\begin{equation}
\mathcal{F}^{\ast}\left\{  H\left(  \mathbf{v}\right)  \delta_{\ln,\infty
}^{\left[  q\right]  }\left(  \mathbf{u}\right)  ;\mathbf{x}\right\}
=\mathfrak{L}_{q}\left\{  H\left(  \mathbf{v}\right)  ;\mathbf{w}\right\}
\delta_{\ast}^{\left[  -n-q\right]  }\left(  \mathbf{x}\right)  \,.
\label{FFpp}%
\end{equation}

\end{example}

\begin{example}
\label{Example 6}We now consider the transform of the general thick deltas
$f\left(  \mathbf{w}\right)  \delta_{\ast}^{\left[  m\right]  }\left(
\mathbf{x}\right)  .$ Let $m=-n-q.$ Different formulas arise depending on $m$
and $q.$ If $1-n\leq m,q\leq-1$\ then $\mathcal{D}_{q}^{^{\prime}}%
=\mathcal{D}_{m}^{^{\prime}}=\mathcal{D}^{^{\prime}}$ so that inversion of
(\ref{FFp 1}), remembering (\ref{a 28}),\ gives
\begin{equation}
\mathcal{F}_{\ast}\left\{  f\left(  \mathbf{w}\right)  \delta_{\ast}^{\left[
m\right]  }\left(  \mathbf{x}\right)  ;\mathbf{u}\right\}  =\mathfrak{K}%
_{m}\left\{  f\left(  \mathbf{w}\right)  ;\mathbf{v}\right\}  \delta_{\infty
}^{\left[  -n-m\right]  }\left(  \mathbf{u}\right)  \,. \label{FFp 2}%
\end{equation}
If $m\geq0,$ that is $q\leq-n,$ we decompose $f\in\mathcal{D}^{\prime}\left(
\mathbb{S}\right)  $ as $f=p_{m}+f_{m}$ where $f_{m}\in\mathcal{D}%
_{q}^{^{\prime}}=\mathcal{D}_{m}^{^{\prime}}$ and $p_{m}\in\mathcal{P}_{m}.$
We now notice that $\mathcal{F}_{\ast}\left(  p\delta_{\ast}^{\left[
m\right]  }\right)  $ is the finite part regularization $\mathcal{P}%
f_{\mathcal{W}}\left(  P_{m}\left(  \mathbf{u}\right)  \right)  $ of a
homogeneous polynomial of degree $m,$ namely $P_{m}=\mathcal{F}\left(
\Pi_{\mathcal{S}_{\ast}^{\prime},\mathcal{S}^{\prime}}\left(  f\delta_{\ast
}^{\left[  m\right]  }\right)  \right)  ,$ obtaining%
\begin{equation}
\mathcal{F}_{\ast}\left\{  f\left(  \mathbf{w}\right)  \delta_{\ast}^{\left[
m\right]  }\left(  \mathbf{x}\right)  ;\mathbf{u}\right\}  =\mathcal{P}%
f_{\mathcal{W}}\left(  P_{m}\left(  \mathbf{u}\right)  \right)  +\mathfrak{K}%
_{m}\left\{  f_{m}\left(  \mathbf{w}\right)  ;\mathbf{v}\right\}
\delta_{\infty}^{\left[  -n-m\right]  }\left(  \mathbf{u}\right)  \,.
\label{FFp 2p}%
\end{equation}
In particular, when $m=0,$ since $\mathcal{F}_{\ast}\left(  \delta_{\ast
}^{\left[  0\right]  }\right)  =\mathcal{P}f_{\mathcal{W}}\left(  1\right)  ,$
we obtain%
\begin{equation}
\mathcal{F}_{\ast}\left\{  f\left(  \mathbf{w}\right)  \delta_{\ast}^{\left[
0\right]  }\left(  \mathbf{x}\right)  ;\mathbf{u}\right\}  =M\mathcal{P}%
f_{\mathcal{W}}\left(  1\right)  +\mathfrak{K}_{0}\left\{  f\left(
\mathbf{w}\right)  -M;\mathbf{v}\right\}  \delta_{\infty}^{\left[  -n\right]
}\left(  \mathbf{u}\right)  \,, \label{FFp 3}%
\end{equation}
where $M$ is the constant $M=\left(  1/C\right)  \left\langle f\left(
\mathbf{w}\right)  ,1\right\rangle .$

Finally, if $m\leq-n,$ i.e. $q\geq0,$ the decomposition $f=p_{m}+f_{m}$ where
$f_{m}\in\mathcal{D}_{q}^{^{\prime}}=\mathcal{D}_{m}^{^{\prime}}$ and
$p_{m}\in\mathcal{P}_{-n-m}=\mathcal{P}_{q}$ yields%
\begin{align}
\mathcal{F}_{\ast}\left\{  f\left(  \mathbf{w}\right)  \delta_{\ast}^{\left[
m\right]  }\left(  \mathbf{x}\right)  ;\mathbf{u}\right\}   &  =\left(
2\pi\right)  ^{n}\mathfrak{L}_{-n-m}^{-1}\left\{  p_{m}\left(  \mathbf{w}%
\right)  ;-\mathbf{v}\right\}  \delta_{\ln,\infty}^{\left[  -n-m\right]
}\left(  \mathbf{u}\right) \label{FFp 4}\\
& \quad +\mathfrak{K}_{m}\left\{  f_{m}\left(  \mathbf{w}\right)  ;\mathbf{v}%
\right\}  \delta_{\infty}^{\left[  -n-m\right]  }\left(  \mathbf{u}\right)
\,.\nonumber
\end{align}

\end{example}

We will consider the thick Fourier transform of several other distributions in
forthcoming papers.

\appendix

\section{Guide to notation\label{Appendix}}

{\small \noindent\textbf{Constants} }

{\small $c_{n.m},$ \ \ \ \ \ \ \ \hfill equation (\ref{ce}) }

{\small $C,$ surface area of the unit sphere,\ \ \ \ \ \ \ \hfill equation
(\ref{ce}) }

{\small \noindent\textbf{Spaces} }

{\small $\mathcal{D}_{\ast,\mathbf{a}}\left(  \mathbb{R}^{n}\right)  ,$
$\mathcal{D}_{\ast}\left(  \mathbb{R}^{n}\right)  ,$ \ \ \ \ \ \ \ \hfill
paragraph before Definition \ref{Definition A} }

{\small $\mathcal{D}_{\ast,\mathbf{a}}^{^{\prime}}\left(  \mathbb{R}%
^{n}\right)  ,$ $\mathcal{D}_{\ast}^{^{\prime}}\left(  \mathbb{R}^{n}\right)
,$ \ \ \ \ \ \ \ \hfill Definition \ref{Definition A} }

{\small $\mathcal{A}_{\ast,\mathbf{a}}\left(  \mathbb{R}^{n}\right)  ,$ in
particular $\mathcal{S}_{\ast,\mathbf{a}}\left(  \mathbb{R}^{n}\right)
,$\ \ \ \ \ \ \ \hfill Definition \ref{Def OS.1} }

{\small $\mathcal{A}_{\ast,\mathbf{a}}^{\prime}\left(  \mathbb{R}^{n}\right)
,$ in particular $\mathcal{S}_{\ast,\mathbf{a}}^{\prime}\left(  \mathbb{R}%
^{n}\right)  ,$\ \ \ \ \ \ \ \hfill Definition \ref{Def OS.1} }

{\small $\mathcal{X}_{q},$ in particular $\mathcal{D}_{q}$ and $\mathcal{D}%
_{q}^{\prime},$ \ \ \ \ \ \ \hfill Subsection
\ref{Subsection: the operators k and l} }

{\small $\mathcal{P}_{q},$ polynomials of degree $q$ on the sphere,
\ \ \ \ \ \hfill Subsection \ref{Subsection: the operators k and l} }

{\small $\mathcal{S}_{\ast,\text{reg}}\left(  \mathbb{R}^{n}\right)  ,$ \hfill
paragraph after equation (\ref{b 1}) }

{\small $\mathcal{W}_{\mathrm{pre}}\left(  \mathbb{R}^{n}\right)  ,$ \hfill
Definition \ref{Definition b 1} }

{\small $\mathcal{W}\left(  \mathbb{R}^{n}\right)  ,$ \hfill Definition
\ref{Definition W 1} }

{\small \noindent\textbf{Operators} }

{\small $\mathcal{K}_{\beta}\left\{  a\left(  \mathbf{w}\right)
;\mathbf{v}\right\}  =\left\langle K_{\beta}\left(  \mathbf{w},\mathbf{v}%
\right)  ,a\left(  \mathbf{w}\right)  \right\rangle _{\mathbf{w}},$
\ \ \ \ \ \ \ \ \ \hfill equation (\ref{a.13p}) }

{\small $\mathcal{L}_{q}\left\{  a\left(  \mathbf{w}\right)  ;\mathbf{v}%
\right\}  =\left\langle L_{q}\left(  \mathbf{w},\mathbf{v}\right)  ,a\left(
\mathbf{w}\right)  \right\rangle _{\mathbf{w}},$ \ \ \ \ \ \ \ \ \ \hfill
equation (\ref{a.16p}) }

{\small $\mathfrak{K}_{q}\left\{  a\left(  \mathbf{w}\right)  ;\mathbf{v}%
\right\}  ,$ \ \ \ \ \ \ \ \ \ \ \hfill Subsection
\ref{Subsection: the operators k and l} }

{\small $\mathfrak{L}_{q}\left\{  a\left(  \mathbf{w}\right)  ;\mathbf{v}%
\right\}  ,$ \ \ \ \ \ \ \ \ \ \ \hfill Subsection
\ref{Subsection: the operators k and l} }

{\small delta part, \hfill Subsection
\ref{Subsection:Delta parts and polynomial parts} }

{\small ordinary part, \hfill Subsection
\ref{Subsection:Delta parts and polynomial parts} }

{\small polynomial part, \hfill Subsection
\ref{Subsection:Delta parts and polynomial parts} }

{\small polynomial free part, \hfill Subsection
\ref{Subsection:Delta parts and polynomial parts} }

{\small $\mathcal{F}_{\text{t}}^{\ast}\left\{  \Phi\left(  \mathbf{u}\right)
;\mathbf{x}\right\}  $ and $\mathcal{F}_{\ast,\text{t}}\left\{  \Phi\left(
\mathbf{u}\right)  ;\mathbf{x}\right\}  ,$ Fourier transform of test
functions,\hfill equation \eqref{W 3} }

{\small $\tau_{A}^{\mathcal{W}}$ and $\tau_{A}^{\mathcal{W}^{\prime}},$ \hfill
Subsection \ref{Subsection: Linear changes of variables in} }

{\small $M_{u_{j}}^{\mathcal{W}}$ and $M_{u_{j}}^{\mathcal{W}^{\prime}},$
\hfill Subsection \ref{Subsection: Multiplication by polynomials in} }

{\small $\mathcal{F}^{\ast}\left\{  f\left(  \mathbf{x}\right)  ;\mathbf{u}%
\right\}  ,$ Fourier transform of thick distributions,\hfill equation
(\ref{W 4}) }

{\small $\mathcal{F}_{\ast}\left\{  g\left(  \mathbf{u}\right)  ;\mathbf{x}%
\right\}  ,$ Fourier transform of $sl-$thick distribution,\hfill equation
(\ref{W 5}) }

{\small \noindent\textbf{Distributions} }

{\small $g\left(  \mathbf{w}\right)  \delta_{\ast}^{\left[  q\right]  }\left(
\mathbf{x-a}\right)  ,$ \ \ \ \ \ \ \ \ \ \ \ \hfill equation (\ref{thick}) }

{\small $\mathcal{P}f\left\{  g\left(  \mathbf{z}\right)  ;\mathbf{x}\right\}
=\mathcal{P}f\left(  g\right)  ,$ \ \ \ \ \ \ \ \ \ \hfill Definition
\ref{Definition Finite Part} }

{\small $G\left(  \mathbf{v}\right)  \delta_{\infty}^{\left[  q\right]  },$
\hfill Definition \ref{Def Wp 1} }

{\small $H\left(  \mathbf{v}\right)  \delta_{\ln,\infty}^{\left[  q\right]
},$ \hfill Definition \ref{Def Wp 1} }

{\small $\mathcal{P}f_{\mathcal{W}}\left(  g\right)  ,$ \hfill paragraph after
Definition \ref{Def Wp 1}}

\end{document}